\documentclass[12pt, leqno]{article}
\usepackage{amsmath}
\usepackage{amsthm}
\usepackage{amssymb}
\usepackage{latexsym}
\usepackage{graphicx}
\allowdisplaybreaks[1]
\usepackage{amsfonts}
\usepackage{ascmac}
\usepackage{color}
\usepackage{enumerate}

\theoremstyle{definition}
\theoremstyle{plain}
\allowdisplaybreaks[1]
\date{}
\newtheorem{Thm}{Theorem}[section]
\newtheorem{Prop}[Thm]{Proposition}
\newtheorem{Lemma}[Thm]{Lemma}
\newtheorem{Cor}[Thm]{Corollary}

\newcommand{\p}{\partial}

\newcommand{\dis}{\displaystyle}

\newcommand{\norm}{\parallel}

\newcommand{\Z}{{\mathbb Z}}
\newcommand{\Q}{{\mathbb Q}}

\newcommand{\N}{{\mathbb N}}
\newcommand{\R}{{\mathbb R}}

\newcommand{\ep}{\varepsilon }
\newcommand{\2}{\frac{1}{2} }
\newcommand{\wto}{\rightharpoonup}
\newcommand{\Omegain}{\Omega_h\setminus\p\Omega_h}
\newcommand{\tOmegain}{{\tilde{\Omega}_h\setminus\p\tilde{\Omega}_h}}
\newcommand{\tOmega}{{\tilde{\Omega}}}
\newcommand{\nnorm}{|\!|\!|}

\def\text#1{\mbox{#1 }}

\title{\bf A finite difference method for inhomogeneous incompressible Navier-Stokes equations}
\author{Kohei Soga
\footnote{Department of Mathematics, Faculty of Science and Technology, Keio University, 3-14-1 Hiyoshi, Kohoku-ku, Yokohama, 223-8522, Japan. E-mail:  soga@math.keio.ac.jp 
}}
\begin{document}
\maketitle
\begin{abstract} 
\noindent  This paper provides mathematical analysis of an elementary fully discrete finite difference method applied to inhomogeneous (non-constant density and viscosity) incompressible Navier-Stokes system on a bounded domain. The proposed method consists of a version of Lax-Friedrichs explicit scheme for the transport equation and a version of Ladyzhenskaya's implicit scheme for the Navier-Stokes  equations. Under the condition that the initial density profile is strictly away from $0$, the scheme is proven to be strongly convergent to a weak solution (up to a subsequence) within an arbitrary time interval, which can be seen as a proof of existence of a weak solution to the system. The results contain a new  Aubin-Lions-Simon type compactness method with an interpolation inequality between strong norms of the velocity and a weak norm of the product of the density and velocity.     
\medskip

\noindent{\bf Keywords:}  inhomogeneous incompressible Navier-Stokes equations; transport equation; weak solution; finite difference method  \medskip

\noindent{\bf AMS subject classifications:}   35Q30; 35Q49;   35D30; 65M06
\end{abstract}
%
\setcounter{section}{0}
\setcounter{equation}{0}
\section{Introduction}
We consider the inhomogeneous incompressible Navier-Stokes equations on a general bounded domain of $\R^3$, i.e., the standard model of a mixture of miscible incompressible fluids with different densities and the non-constant viscosity,   
\begin{eqnarray}\label{NS}
 && \left\{
\begin{array}{lll}
\qquad\qquad\,\, \p_t\rho+v\cdot \nabla \rho &=&0 \mbox{\quad\quad\quad\quad\quad\quad\quad\quad\,\, in $(0,T]\times\Omega $,}
\medskip\\
\qquad \rho\Big(\p_tv+(v\cdot \nabla)v \Big)&=& \nabla\cdot\{\mu(\rho) (\nabla v+ {}^{\rm t}\!(\nabla v))\}+\rho f -\nabla p\\
&&\mbox{\qquad\quad\,\,\,\,\qquad\quad\quad\quad\, in $(0,T]\times\Omega $,}
\medskip\\
\qquad\qquad\qquad\,\,\quad\nabla\cdot v &=&0\mbox{\quad\quad\quad\quad\quad\quad\quad\quad\,\, in $(0,T]\times\Omega$, \qquad\qquad\quad}
\medskip\\
\,\,\,\qquad\qquad\quad\quad\,\,\, v(0,\cdot)&=&v^0\mbox{\quad\qquad\qquad\qquad\,\,\,\,\,\, in $\Omega$},
\medskip\\
\qquad\qquad\qquad \,\,\,\,\,\,\rho(0,\cdot)&=&\rho^0\mbox{\quad\qquad\qquad\qquad\,\,\,\,\,\, in $\Omega$},
\medskip\\
\qquad\qquad \qquad \qquad \,\,\, \,\,\,v&=&0\mbox{\quad\qquad\qquad\qquad\,\,\,\,\,\,\,\, on $(0,T]\times\partial \Omega$},
\end{array}
\right.\\\nonumber  
&&\Omega\subset \R^3 \mbox{ is a bounded connected open set  with a Lipschitz boundary,}
\end{eqnarray}
where $v=v(t,x)$ is the unknown  velocity, $\rho=\rho(t,x)$ is the unknown density,  $p=p(t,x)$ is the unknown pressure, $\mu(\cdot)$ is a given viscosity function depending on the density, $f=f(t,x)$ is a given external force, 
$T$ is an arbitrary positive terminal time, $v^0$ and $\rho^0$ are initial data, $\nabla=(\p_{x_1},\p_{x_2},\p_{x_3})$, $\Delta=\p_{x_1}^2+\p_{x_2}^2+\p_{x_3}^2$,  $v_t=\partial_t v$, $v_{x_j}=\partial_{x_j}v$, etc., stand for  the partial (weak) derivatives of $v(t,x)$, $\nabla v$ is the Jacobian matrix of $v$ and  
$ {}^{\rm t}\!(\nabla v)$ stands for the transpose of $\nabla v$.
In this paper, we suppose that $\mu$, $f$, $v^0$ and $\rho^0$ are  such that   
\begin{eqnarray*}
&&\mu:[0,\infty)\to(0,\infty),\quad\mbox{continuous},\\
&&\mbox{$f\in L^2_{\rm loc}([0,\infty);L^2(\Omega)^3)$;  $v^0\in L^2(\Omega)^3$; $\rho^0\in L^\infty(\Omega)$ with $\dis \inf_\Omega\rho^0>0$},
\end{eqnarray*} 
where $v^0$ does not need to be from $L^2_{\sigma}(\Omega)$. Here, $C^r_{0}(\Omega)=C^r_{0}(\Omega;\R)$  is the family of $C^r$-functions\,:\,$\Omega\to\R$ that are equivalently $0$ near $\p\Omega$; 
$C^r_{0,\sigma}(\Omega):=\{v\in C^r_0(\Omega)^3\,|\,\nabla\cdot v=0\}$;  
$L^2(\Omega)=L^2(\Omega;\R)$; 
$H^1_0(\Omega)=H^1_0(\Omega;\R)$ is the closure of $C^\infty_0(\Omega)$ with respect to the norm  $\norm \cdot\norm_{H^1(\Omega)}$; $L^2_{\sigma}(\Omega)$ (resp. $H^1_{0,\sigma}(\Omega)$)  is the closure of $C^\infty_{0,\sigma}(\Omega)$ with respect to the norm $\norm\cdot\norm_{L^2(\Omega)^3}$ (resp. $\norm\cdot\norm_{H^1(\Omega)^3}$);  $\tilde{H}^1_{0,\sigma}(\Omega):=\{v\in H^1_0(\Omega)^3\,|\,\nabla\cdot v=0 \} $, where  $\tilde{H}^1_{0,\sigma}(\Omega)$ coincides with $H^1_{0,\sigma}(\Omega)$ provided $\p \Omega$ is Lipschitz (see, e.g., Theorem 1.6 and Remark 1.7 of Chapter 1 in \cite{Temam-book}); $x\cdot y:=\sum_{i=1}^3x_iy_i$ for $x,y\in\R^3$. 

If $v$ and $\rho$ are  smooth, the first, second and third equations of \eqref{NS} yield  
\begin{eqnarray}\label{(1.1)b}
&& \p_t\rho+ \nabla\cdot(\rho v)=0,\\
\label{(1.1)a}
&&\p_t(\rho v)+\sum_{j=1}^3\p_{x_j}(\rho v_jv)=\nabla\cdot\{\mu(\rho) (\nabla v+{}^{\rm t}\!(\nabla v))\}+\rho f-\nabla p.  
\end{eqnarray}
This leads to the following definition of a weak solution of \eqref{NS}: 
 {\it a pair of functions  $\rho$ and $v$ is called a  weak solution of \eqref{NS}, if 
\begin{eqnarray}\nonumber
&&\rho\in L^\infty([0,T];L^\infty(\Omega))\mbox{ with $\rho>0$},\\\nonumber
&&v\in L^2([0,T];H^1_{0,\sigma}(\Omega))\cap L^\infty([0,T];L^2(\Omega)^3),\quad \\\label{1313}
&&\int_\Omega \rho^0(x)\varphi(0,x)dx+\int_0^T\int_\Omega \Big(\rho(t,x)\p_t\varphi(t,x)+v(t,x)\rho(t,x)\cdot \nabla\varphi(t,x)\Big)dxdt=0,\\\nonumber
&& \mbox{$\forall\,\varphi\in C^\infty([0,T]\times\R^3;\R)$ with {\rm supp}$(\varphi)\subset[0,T)\times\R^3$ compact};\\
\label{weak-form-NS}
&&\int_\Omega \rho^0(x)v^0(x)\cdot\phi(0,x)dx+ \int_0^T\int_\Omega \rho(t,x)v(t,x)\cdot \partial_t\phi(x,t)dxdt \\\nonumber
&&\quad +\sum_{j=1}^3\int_0^T\int_\Omega \rho(t,x)v_j(t,x)v(t,x)\cdot \partial_{x_j}\phi(t,x)dxdt\\\nonumber
&&\quad  -\sum_{j=1}^3\int_0^T\int_\Omega\mu(\rho(t,x))(\partial_{x_j}v(t,x)+\nabla v_j(t,x) )\cdot\partial_{x_j}\phi(t,x)dxdt \\\nonumber
&&\quad +\int_0^T\int_\Omega \rho(t,x)f(t,x)\cdot\phi(t,x)dxdt=0,\\\nonumber
&& \mbox{$\forall\,\phi\in C^\infty([0,T]\times\Omega;\R^3)$ with {\rm supp}$(\phi)\subset[0,T)\times\Omega$ and $\nabla\cdot \phi =0$}.
\end{eqnarray}}
\noindent We remark that the choice of $\varphi$  in \eqref{1313} implies that the $0$-extensions  of $\rho^0$, $\rho$ and $v$ outside $\Omega$ provide the unique DiPerna-Lions weak solution of \eqref{(1.1)b} in $(0,T]\times\R^3$ obtained in \cite{DiPerna-Lions}; hence $\rho$ belongs to $C([0,T];L^p(\Omega))$ for all $p\in[1,\infty)$ (see Introduction of \cite{Soga}).  

Existence of a weak solution to \eqref{NS} was established by  Antontsev-Kazhikhov \cite{AK} and Kazhikhov \cite{Kazhikhov} based on a Galerkin method under the assumption that the initial density profile is strictly positive and the viscosity is constant. Then, with finer a priori estimates,  Kim \cite{Kim} and Simon \cite{Simon} removed the positivity assumption, where the $L^\infty([0,T];L^2(\Omega)^3)$-regularity of the velocity was missing;  Lions \cite{Lions} allowed the non-constant viscosity. We refer to Danchin-Mucha \cite{DM} for further developments and  reviews of mathematical analysis of \eqref{NS} including its strong solutions. 

In regards to mathematical analysis of  numerical methods for \eqref{NS}, Liu-Walkington \cite{LW} proposed a numerical scheme that was strongly convergent to a weak solution based on a discontinuous Galerkin method for \eqref{(1.1)b} and a finite element method for \eqref{(1.1)a}, where they supposed  the positivity condition for the density  but allowed the non-constant viscosity.   Guermond-Salgado \cite{GS} demonstrated error analysis of a Galerkin type numerical method applied to \eqref{NS} (with strictly positive density and the constant viscosity coefficient) assuming existence of a smooth solution.

The purpose of this paper is to provides an elementary but rigorous approach to the existence of weak solutions of  \eqref{NS} based on a very simple finite difference scheme 
(we postpone actual implementation of the scheme for numerical tests). We are inspired by a finite difference scheme applied to homogeneous incompressible Navier-Stokes equations, and therefore we give a brief overview of the development of finite difference methods in the homogeneous case.  
In the huge literature of homogeneous incompressible Navier-Stokes equations, there are a number of results on mathematical analysis of various numerical methods. 
Among them, finite difference methods seem to be more   elementary and direct to the exact differential equations 
than other major methods.  
To the best of author's knowledge, the first rigorous treatment of fully discrete finite difference approximation of the homogeneous incompressible Navier-Stokes equations was given by Krzywicki-Ladyzhenskaya  \cite{KL} and Ladyzhenskaya \cite{Ladyzhenskaya} (here, we call it {\it Ladyzhenskaya's scheme}), where they proposed an elementary fully discrete implicit finite difference scheme on the uniform Cartesian grid to discretizes the homogeneous Navier-Stokes equations including the pressure and the divergence-free constraint. In \cite{Ladyzhenskaya}, she showed its solvability and a priori estimates; although she skipped details of its strong convergence to a Leray-Hopf weak solution, the issue turned out to be rather delicate, i.e., some ``equi-continuity'' with respect to the time variable or so-called the Aubin-Lions-Simon compactness method is  necessary (see, e.g., \cite{JL} and \cite{Temam-book}).       
Chorin \cite{Chorin} modified Ladyzhenskaya's scheme by separating the step of realizing the (discrete) divergence-free constraint from the discrete time evolution, where he demonstrated a convergence proof and error estimates of the scheme assuming a smooth exact solution on a $2$ or $3$-dimensional torus.   
Temam \cite{Temam-2} also investigated this type of fully discrete scheme based on a framework of finite element methods.  
Their methods are nowadays called  projection methods and many versions are known.   
 Kuroki-Soga \cite{Kuroki-Soga} proved convergence of (slightly modified) Chorin's original scheme to a Leray-Hopf weak solution by adjusting Aubin-Lions-Simon compactness arguments to space-time step functions with the discrete divergence-free constraint and the discrete time-differentiation, where difficulty comes from the fact that the discrete divergence-free constraint and the discrete time-differentiation vary according to the mesh size (one cannot work only within $C^\infty_{0,\sigma}$ or $H^1_{0,\sigma}$). 
 
In this paper, we employ a version of Ladyzhenskaya's scheme to \eqref{(1.1)a}, not a projection method. The advantage to do so is that Ladyzhenskaya's scheme provides a discrete velocity field possessing both the discrete divergence-free constraint and a good (discrete) $L^2_tH^1_x$-bound. Note that Chorin's scheme does not have such a feature (see \cite{Kuroki-Soga}). As for \eqref{(1.1)b}, we use a Lax-Friedrichs type explicit scheme. Our combination of the two schemes, which is probably the simplest method to solve  \eqref{NS},  must overcome the following difficulties in order to achieve  strong convergence to a weak solution: 
\begin{enumerate}
\item[(D1)]  The velocity field $v$ in \eqref{(1.1)b} can be unbounded and verification of the CFL-condition for the Lax-Friedrichs explicit scheme is non-trivial.
\item[(D2)] Aubin-Lions-Simon type compactness arguments to prove strong convergence of  the approximate velocity field refer to its discrete time-derivative, but controllability of the discrete time-derivative of the velocity field through   \eqref{(1.1)b}  is not clear, i.e., what we actually have is the discrete time-derivative of [density]$\times$[velocity].
\end{enumerate}
\indent An idea to overcome (D1) was given by Soga \cite{Soga}, where he showed a new technique to deal with the transport equation with an unbounded Sobolev velocity field through the  Lax-Friedrichs type explicit scheme, introducing  the generalized hyperbolic scale (see \eqref{scale} below) and truncation of the velocity field together with a suitable measure estimate for the truncated part. The direct consequence of this method is weak convergence to a DiPerna-Lions weak solution obtained in  \cite{DiPerna-Lions}, but a fine estimate of the norm of approximate solutions implies that the weak convergence is in fact strong convergence (it is essential that a DiPerna-Lions weak solution conserves its $L^p_x$-norm). We will follow this idea to deal with \eqref{(1.1)b}, where local averaging of possibly unbounded velocity fields is used instead of the truncation used in \cite{Soga}  in order to keep the discrete divergence-free constraint; an artificial boundary condition is imposed to the discretization of \eqref{(1.1)b}; the artificial boundary condition does not cause any harm  to the solution, if the  (locally averaged) velocity field vanishes on the boundary; since the  support of the locally averaged velocity field can be slightly larger than $\Omega$,  \eqref{(1.1)b} will be solved on a domain larger than $\Omega$ with constant-extension of the velocity field and the density field.         
      
(D2) will be overcome by modification of the interpolation inequality for the discrete velocity field  obtained by Kuroki-Soga \cite{Kuroki-Soga} in such a way that the ``weak norm'' of the velocity field is replaced by that of [density]$\times$[velocity] (see Lemma \ref{key-lemma} below); this is possible as long as the density is positive almost everywhere. In the end, we will see that the whole reasoning is quite similar to the homogeneous case.  

It is an open question how to treat the case with vacuum ($\inf \rho^0=0$) in our framework (strong convergence of the approximate velocity field is not clear). It would be also interesting to place our finite difference framework in the context of compressible problems, where we refer to \cite{Karper}, \cite{Feireisl2},  \cite{Gallouet2} and \cite{HS} for recent developments of mathematical analysis of numerical methods for compressible Navier-Stokes equations.  

Section 2 provides the notation and basic calculus on  the uniform Cartesian grid. Section 3 discusses the unique solvability and a priori estimates of the discrete problem. Section 4 demonstrates convergence of our scheme.     
\setcounter{section}{1}
\setcounter{equation}{0}
\section{Preliminary}

Consider the grid $h\Z^3:=\{ (hz_1,hz_2,hz_3)\,|\,z_1,z_2,z_3\in\Z \}$ with the mesh size $h>0$. Let $e^1,e^2,e^3$ be the standard basis of $\R^3$. 
The boundary of  $G\subset h\Z^3$ is defined as $\partial G:=\{ x\in G\,|\,\{x\pm he^i\,|\,i=1,2,3\}\not\subset G \}$. 

Let $\Omega$ be a bounded, open, connected subset of $\R^3$ with a Lipschitz boundary $\partial\Omega$. Set 
\begin{eqnarray*}
C_h(x)&:=&\Big[x_1-\frac{h}{2},x_1+\frac{h}{2}\Big)\times\Big[x_2-\frac{h}{2},x_2+\frac{h}{2}\Big)\times\Big[x_3-\frac{h}{2},x_3+\frac{h}{2}\Big),\\
C_h^+(x)&:=&C_h\Big(x+\frac{h}{2}e^1+\frac{h}{2}e^2+\frac{h}{2}e^3\Big)=[x_1,x_1+h)\times[x_2,x_2+h)\times[x_3,x_3+h).
\end{eqnarray*}
We discretize \eqref{(1.1)a} on the set 
$$\Omega_h:=\{ x\in \Omega\cap h\Z^3\,|\,\,\,\,C_{4h}(x)\subset \Omega\}.$$ 
For technical reasons (we will see them later), we solve  \eqref{(1.1)b} on a  domain slightly larger than $\Omega$: let $\tilde{\Omega}\subset \R^3$ be a connected bounded  open set such that 
\begin{eqnarray}\label{2epepe}
\tilde{\Omega}\supset \bigcup_{x\in\Omega} \{ y\in\R^3\,|\, |y-x|\le \epsilon_0\}\quad \mbox{($\epsilon_0>0$ is a constant)}.
\end{eqnarray}    
We discretize \eqref{(1.1)b} on the set 
$$\tilde{\Omega}_h:=\{ x\in \tilde{\Omega}\cap h\Z^3\,|\,\,\,\,C_{4h}(x)\subset \tilde{\Omega}\},$$ 
where we always assume that $h\ll \epsilon_0$. 

Define the discrete derivatives of a function $\phi:G\to\R$ with $G\subset h\Z^3$ as
\begin{eqnarray*}
&&D_i^+\phi(x):=\frac{\phi(x+he^i)-\phi(x)}{h},\,\,\,D_i^-\phi(x):=\frac{\phi(x)-\phi(x-he^i)}{h},\\
&&D_i\phi(x):=\frac{\phi(x+he^i)-\phi(x-he^i)}{h}
\end{eqnarray*}  
for each $x\in G$,  where {\it we always assume that $\phi$ is extended outside $G$ in a certain way, i.e., $\phi(x\pm he^i)$ are given even  if $x\pm he^i\not\in G$; in particular, if $\phi|_{\p G}=0$, we take the $0$-extension}. 
For $x,y\in\R^d$, set $x\cdot y:=\sum_{i=1}^dx_iy_i$, $|x|:=\sqrt{x\cdot x}$. Define the discrete gradient  and the discrete divergence  for  functions $\phi:G\to\R$ and $w=(w_1,w_2,w_3):G\to\R^3$  as 
\begin{eqnarray*}
&&D\phi(x):=(D_1\phi(x),D_2\phi(x),D_3\phi(x)),\,\,\,
D^\pm\phi(x):=(D^\pm_1\phi(x),D^\pm_2\phi(x),D^\pm_3\phi(x)),\\
&&D\cdot w(x):=D_1w_1(x)+D_2w_2(x)+D_3w_3(x),\\
&&D^\pm\cdot w(x):=D^\pm_1w_1(x)+D^\pm_2w_2(x)+D^\pm_3w_3(x)
\end{eqnarray*}
for each $x\in G$. 
 We often use the summation by parts such as  
\begin{eqnarray}\label{by-parts}
  &&\sum_{x\in G}w(x) D_i^+\phi(x)=-\sum_{x\in G }D^-_i w(x) \phi(x),\,\,\,
\sum_{x\in G}w(x) D_i\phi(x)=-\sum_{x\in G}D_i w(x) \phi(x)
\end{eqnarray}
for functions $w,\phi:G\to\R$ that are extended to be $0$ outside $G$.   

Define the discrete $L^p$-norms of a function $\phi:G\to\R$ or $\R^3$  with $G\subset h\Z^3$ as 
\begin{eqnarray*}
\norm \phi\norm_{p,G}:=\Big(\sum_{x\in G}|\phi(x)|^ph^3 \Big)^{\frac{1}{p}},\,\,\,\,\norm \phi\norm_{\infty,G}:=\max_{x\in G}|\phi(x)| ;
\end{eqnarray*}
in particular for $p=2$, we introduce the discrete inner product as  
\begin{eqnarray*}
(\phi,\tilde{\phi})_{G}:=\sum_{x\in G}\phi(x)\tilde{\phi}(x)h^3,\quad \norm \phi\norm_{2,G}=\sqrt{(\phi,\phi)_{G}}. 
\end{eqnarray*} 
\indent  We introduce a local averaging operator $\mathcal{A}_h^{k}$, which plays on $h\Z^3$ like the mollifier in $\R^3$. For each $k\in\N \cup \{0\}$, define the set 
$$A_h^k:=\Big\{ y=(y_1,y_2,y_3)\in \R^3\,\Big|\, |y_i|\le \frac{h}{2}+kh,\,\,\, i=1,2,3 \Big\}.$$ 
For each function $\phi:\Omega_h\to\R$ or $\R^3$, extend $\phi$ to be $0$ outside $\Omega_h$ and define the locally averaged function $\mathcal{A}_h^k\phi:h\Z^3\to\R$ or $\R^3$ as 
$$\mathcal{A}_h^{k}\phi(x):=\frac{1}{{\rm vol}(A_h^k)}\sum_{y\in A_h^k\cap h\Z^3}\phi(x+y) h^3,$$
where $\sum_{y\in A_h^k\cap h\Z^3}h^3={\rm vol}(A_h^k)$; in particular $\mathcal{A}_h^{0}\phi=\phi$ and  $\norm \mathcal{A}_h^{k}\phi\norm_{\infty ,h\Z^3}\to 0$ as $k\to\infty$. An easy calculation shows that  
\begin{eqnarray}\label{molli}
\norm \mathcal{A}_h^{k}\phi\norm_{p,h\Z^3}\le\norm \phi\norm_{p,\Omega_h},\quad\forall\,p\in[1,\infty]. 
\end{eqnarray}
In fact, the case of $p=1,\infty$ is clear; in the case $p\in(1,\infty)$, with $1/p+1/p^\ast=1$, we have by H\"older's inequality,   
\begin{eqnarray*}
&&| \mathcal{A}_h^{k}\phi(x)|\le\frac{1}{{\rm vol}(A_h^k)}\Big( \sum_{y\in A_h^k\cap h\Z^3}1^{p^\ast} h^3\Big)^\frac{1}{p^\ast}\Big(\sum_{y\in A_h^k\cap h\Z^3}|\phi(x+y)|^p h^3 \Big)^\frac{1}{p},\\
&&\sum_{x\in h\Z^3}| \mathcal{A}_h^{k}\phi(x)|^ph^3 \le\frac{1}{{\rm vol}(A_h^k)}\sum_{y\in A_h^k\cap h\Z^3}
\sum_{x\in  h\Z^3}|\phi(x+y)|^p h^3h^3
\le \norm \phi\norm_{p,\Omega_h}^p.
\end{eqnarray*}
\indent For a technical reason (we will see it later), we sometimes need to argue in an inner part of $\Omega_h$. Define  
\begin{eqnarray*}
\Omega_h^\circ:=\Big\{x\in\Omegain\,\Big|\,x+a^1he^1+a^2he^2+a^3he^3\in\Omegain, \,\,\,a^1,a^2,a^3\in\{0,1,2\}\,     \Big\}.
\end{eqnarray*}
When the central difference $D$ is used, we must look at  the $\{2e^1,2e^2,2e^3\}$-translation invariant subsets $G^1,\ldots,G^8$ of the grid $h\Z^3$, i.e.,   $G^1,\ldots,G^8$ are the sets of grid points with index $($even$,$ even$,$ even$)$, $($even$,$ even$,$ odd$)$, $($even$,$ odd$,$ even$)$, $($odd$,$ even$,$ even$)$, $($even$,$ odd$,$ odd$)$, $($odd$,$ odd$,$ even$)$, $($odd$,$ even$,$ odd$)$, $($odd$,$ odd$,$ odd$)$, respectively. In particular,  the $0$-mean value condition of $\phi:\Omega_h\to\R$ to verify $D\phi=0\Rightarrow \phi=0$  is  given on each $\Omega_h^\circ\cap G^i$.   
We always assume that $h>0$ is small enough so that $\Omega_h^{\circ }\cap G^i$ is connected, i.e., for any $x,\tilde{x}\in\Omega_h^{\circ }\cap G^i$, we have $\omega^1,\omega^2,\ldots,\omega^K\in\{\pm e^i\}_{i=1,2,3}$ such that $x+2h\omega^1+\cdots+2h\omega^k\in\Omega_h^{\circ }\cap G^i$ for all $k\le K$ and $x+2h\omega^1+\cdots+2h\omega^K=\tilde{x}$.
We state a discrete Poincar\'e type inequality:  
\begin{Lemma}[Lemma 2.3 of \cite{Maeda-Soga}]\label{PoinII}
There exists a constant $A>0$ depending only on $\Omega$ for which each function $\phi:\Omega_h\to\R$ satisfies 
\begin{eqnarray*}
&&\sum_{j=1}^8\sum_{x\in\Omega_h^{\circ}\cap G^{j}} |\phi(x)-[\phi]^j|^2h^3
\le A^2 \sum_{x\in\Omega_h\setminus\partial\Omega_h}  |D \phi(x)|^2h^3,\\
&&[\phi]^j:=\Big( \sum_{x\in\Omega_h\cap G^{j}}h^3  \Big)^{-1} \sum_{x\in\Omega_h\cap G^{j}}\phi(x)h^3.
\end{eqnarray*}
\end{Lemma}
\noindent The reason why we use $\Omega_h^{\circ}$ is to avoid the presence of the values of $D \phi$ on $\p\Omega_h$; see the upcoming application of the lemma to the discrete pressure, where the value of its discrete $x$-derivative on  $\p\Omega_h$ is out of any estimate.    

 In order to take out the discrete divergence-free part  of initial data, we need the discrete Helmholtz-Hodge decomposition with the central difference: 
\begin{Lemma}[Theorem 2.4 of \cite{Maeda-Soga}]
 For each function  $u:\Omega_h\to\R^3$, there exist unique functions $w:\Omega_h\to\R^3$ and  $\phi:\Omega_h\to\R$ such that 
\begin{eqnarray}\label{2pre}
&&D\cdot w=0\quad\mbox{ on $\Omega_h$};\qquad
w+D \phi =u\quad \mbox{ on $\Omega_h\setminus \partial\Omega_h$};\\\nonumber 
&&w=0\quad \mbox{ on $\partial\Omega_h$};\quad \sum_{x\in\Omega_h^{\circ }\cap G^j}\phi(x)=0 \quad\mbox{ for } j=1,\cdots,8.
\end{eqnarray}
\end{Lemma}
\noindent The discrete Helmholtz-Hodge decomposition operator $P_h$  for each function $u:\Omega_h\to\R^3$ is defined as  
$$\mbox{$P_hu:=w$ \,\,($w$ is the one obtained in \eqref{2pre}). }$$ 
\indent We state a Korn type inequality.  
\begin{Lemma}\label{Korn}
For each function $w:\Omega_h\to\R^3$ such that $w|_{x\in\p\Omega_h}=0$ with the $0$-extension outside $\Omega_h$, it  holds that 
\begin{eqnarray*}
\sum_{i,j=1}^3\sum_{x\in\Omega_h} (D_j^+w_i(x)+D_i^+w_j(x))^2&=&2\sum_{i,j=1}^3\sum_{x\in\Omega_h}(D_j^+w_i(x))^2+2\sum_{x\in\Omega_h} (D^-\cdot w(x))^2\\
&\ge& 2\sum_{i,j=1}^3\sum_{x\in\Omega_h}(D_j^+w_i(x))^2.
\end{eqnarray*}
\end{Lemma}
\begin{proof}
The assertion follows from 
\begin{eqnarray*}
&&\sum_{i,j=1}^3\sum_{x\in\Omega_h} (D_j^+w_i(x)+D_i^+w_j(x))^2\\
&& =2\sum_{i,j=1}^3\sum_{x\in\Omega_h}(D_j^+w_i(x))^2
+2\sum_{i,j=1}^3\sum_{x\in\Omega_h}D_j^+w_i(x)D_i^+w_j(x), \\
&&\sum_{i,j=1}^3\sum_{x\in\Omega_h}D_j^+w_i(x)D_i^+w_j(x)
=-\sum_{i,j=1}^3\sum_{x\in\Omega_h}w_i(x)D_j^-(D_i^+w_j(x))\\
&& =-\sum_{i,j=1}^3\sum_{x\in\Omega_h}w_i(x)D_i^+(D_j^-w_j(x))
 =\sum_{i,j=1}^3\sum_{x\in\Omega_h}D_i^-w_i(x)D_j^-w_j(x)=\sum_{x\in\Omega_h} (D^-\cdot w(x))^2. 
\end{eqnarray*}
\end{proof}
\setcounter{section}{2}
\setcounter{equation}{0}
\section{Discrete problem}

Let $\tau>0$ be a mesh size for time and let $T_\tau\in\N$ be the discrete terminal time, i.e., $T\in[\tau T_\tau-\tau, \tau T_\tau)$. We sometimes use the notation $t_n:=\tau n$ for $n\in\N\cup\{0\}$. Throughout this paper, we suppose the following generalized hyperbolic scaling condition for the mesh size $(h,\tau)$: 
\begin{eqnarray}\label{scale}
\tau=h^{2-\alpha}\quad  \mbox{with an arbitrarily fixed $\alpha\in(0,1)$}.
\end{eqnarray}
Note that  the necessity of the generalized hyperbolic scaling condition comes only from the explicit scheme for  \eqref{(1.1)b}; $\alpha$ closer to $1$ would cause less numerical diffusivity; $h^{-1+\alpha}$ will be the order of truncation of the possibly $\norm \cdot\norm_{\infty,\Omega_h}$-unbounded discrete velocity fields so that the CFL-condition is valid, where truncation is done by the local averaging $\mathcal{A}_h^{k}$; $\alpha$ closer to $1$ would require larger $k$, which could increase the truncation error.       

Let $f\in L^2_{\rm loc}([0,\infty);L^2(\Omega)^3)$ be a given external force and let  $v^0\in L^2(\Omega)^3$ and $ \rho^0\in L^\infty(\Omega)$  be initial data of \eqref{NS}  satisfying 
$$0<\rho_{\ast}^0\le\rho^0\le \rho^0_{\ast\ast}\quad(\mbox{$\rho_{\ast}^0, \rho_{\ast\ast}^0$ are constants}).$$
We extend $\rho^0$ to $\tilde{\Omega}$ as 
$$\rho^0(x)\equiv \rho^0_{\ast}\quad \mbox{ on $\tilde{\Omega}\setminus\Omega$}$$
Define $f^{n+1}:\Omega_h\to\R^3$ with $n\ge0$, $\eta^0:\tilde{\Omega}_h\to\R$, $u^0:\Omega_h\to\R^3$ and $\tilde{u}^0:h\Z^3\to\R^3$ as 
\begin{eqnarray*}
f^{n+1}(x)&:=&\tau^{-1}h^{-3}\int^{\tau(n+1)}_{\tau n}\int_{C_h^+(x)}f(s,y)dyds,\quad x\in\Omega_h,\,\,\,n\ge0,\\
\eta^0(x)&:=&h^{-3}\int_{C_h^+(x)}\rho^0(y)dy,\quad x\in\tilde{\Omega}_h,\\
u^0(x)&:=&h^{-3}\int_{C_h^+(x)}v^0(y)dy,\quad x\in\Omega_h,\\
\tilde{u}^0&:=&\mathcal{A}_h^{k_0}(P_hu^0)\quad(\mbox{$P_hu^0$ is extended to be $0$ outside $\Omega_h$}) 
\end{eqnarray*}
where $k_0$ is chosen in the following manner: 
\begin{eqnarray*}
k_0=\left\{
\begin{array}{lll} 
&0,\qquad \mbox{if  $\dis \norm P_hu^0\norm_{\infty,\Omega_h}\le \frac{2}{7}h^{-1+\alpha}$}, \\
&\min\{k\in\N \,|\, \norm \mathcal{A}_h^{k}(P_hu^0)\norm_{\infty,h\Z^3}\le \frac{2}{7}h^{-1+\alpha} \},\qquad\mbox{otherwise.}
\end{array}
\right. 
\end{eqnarray*} 
Note that $\eta^0(x)=\rho^0_{\ast}$ on $\p\tilde{\Omega}_h$ due to $h\ll \epsilon_0$ (see \eqref{2epepe}). 

We introduce our discrete problem, which is a system of explicit-implicit recurrence equations. For given $\eta^n:\tilde{\Omega}_h\to[\rho^0_{\ast},\rho^0_{\ast\ast}]$ and $u^n:\Omega_h\to\R^3$ with $u^n=0$ on $\p\Omega_h$ and $D\cdot u^n=0$ on $\Omega_h$  (if $n=0$, the conditions $u^0=0$ on $\p\Omega_h$ and $D\cdot u^0=0$ on $\Omega_h$ are not required), we want to obtain $\eta^{n+1}:\tilde{\Omega}_h\to[\rho^0_{\ast},\rho^0_{\ast\ast}]$, $u^{n+1}:\Omega_h\to\R^3$ and $q^{n+1}:\Omega_h\to\R$ through the following discrete system:   
\begin{eqnarray}\nonumber 
&&B:=\{0,\pm e^1,\pm e^2,\pm e^3\}\quad (\sharp B=7),\\\nonumber 
&&\tilde{u}^n:=\mathcal{A}_h^{k_n}u^n\quad(\mbox{$u^n$ is extended to be $0$ outside $\Omega_h$}), {\quad where}\\\nonumber
&&k_n:=\left\{
\begin{array}{lll} 
&0,\qquad \mbox{if  $\dis \norm u^n\norm_{\infty,\Omega_h}\le \frac{2}{7}h^{-1+\alpha}$}, \\
&\min\{k\in\N \,|\, \norm \mathcal{A}_h^{k}u^n\norm_{\infty,h\Z^3}\le \frac{2}{7}h^{-1+\alpha} \},\qquad \mbox{otherwise}
\end{array}
\right. ,\\\label{d1}
&&\Big(\eta^{n+1}(x)-\frac{1}{7}\sum_{\omega\in B} \eta^n(x+h\omega) \Big)\frac{1}{\tau}+ D\cdot(\eta^n\tilde{u}^n)(x)=0,\quad x\in\tOmegain,\\\label{d2}
&&\eta^{n+1}(x)=\eta^0(x)\,\,(=\rho^0_{\ast}),\quad x\in \partial\tOmega_h,\\\label{d3}
&&\Big(\eta^{n+1}(x)u^{n+1}_i(x)-\frac{1}{7}\sum_{\omega\in B} \eta^n(x+h\omega)u^n_i(x+h\omega) \Big)\frac{1}{\tau}
+D\cdot(\eta^n\tilde{u}^n)(x)u^{n+1}_i(x)\\\nonumber
&&\qquad  +\sum_{j=1}^3\frac{1}{2}\Big(
\eta^n(x-he^j)\tilde{u}^n_j(x-he^j)D_ju^{n+1}_i(x-he^j)\\\nonumber
&&\qquad\qquad\qquad\qquad\qquad +
\eta^n(x+he^j)\tilde{u}^n_j(x+he^j)D_ju^{n+1}_i(x+he^j)\Big)\\\nonumber
&&\quad =D^-\cdot\Big\{\mu(\eta^{n+1})\Big(D^+u^{n+1}_i + D^+_iu^{n+1}\Big)\Big\}(x)+\eta^{n+1}(x)f^{n+1}_i(x) -D_iq^{n+1}(x),\\\nonumber
&&\qquad\qquad\qquad\qquad\qquad\qquad\qquad\qquad\qquad\qquad\qquad x\in \Omegain,\,\,\,\,i=1,2,3,\\\label{d4}
&&u^{n+1}(x)=0,\quad x\in \p\Omega_h,\\\label{d5}
&&D\cdot u^{n+1}(x)=0,\quad x\in \Omega_h.
\end{eqnarray}  
Here are several remarks on the discrete problem: 
\begin{itemize}
\item $\dis \quad D\cdot(\eta^n\tilde{u}^n)(x)=\sum_{j=1}^3\frac{1}{2h}\Big(  \eta^n(x+he^j) \tilde{u}^n_j(x+he^j)-\eta^n(x-he^j)\tilde{u}^n_j(x-he^j)\Big)$,
\begin{eqnarray*}
&& \!\!\!\!D^-\cdot\Big\{\mu(\eta^{n+1})\Big(D^+u^{n+1}_i + D^+_iu^{n+1}\Big)\Big\}(x)\\
&&\quad =\sum_{j=1}^3\frac{1}{h}\Big\{ \mu(\eta^{n+1}(x))\Big(D^+_ju^{n+1}_i(x)+D^+_iu^{n+1}_j(x)\Big) \\
&&\qquad -\mu(\eta^{n+1}(x-he^j))\Big(D^+_ju^{n+1}_i(x-he^j)+D^+_iu^{n+1}_j(x-he^j)  \Big) \Big\}.\qquad\qquad \qquad\qquad 
\end{eqnarray*}
\item  Even if $u^n|_{\p\Omega_h}=0$, we have $\tilde{u}^n|_{\p\Omega_h}\neq0$ in general; if we consider \eqref{d1} on $\Omegain$, the norm of $\eta^{n+1}$ is not controlled properly due to the effect of  $\tilde{u}^n|_{\p\Omega_h}\neq0$; we will see later that $\tilde{u}^n|_{\p\tOmega_h}=0$ for all sufficiently small $(\tau,h)$, which provides a good control of  the norm of $\eta^{n+1}$ and consequently its strong convergence.    
\item If $u^n$ ($n\ge1$) satisfies $u^n=0$ on $\p\Omega_h$ and  $D\cdot u^n=0$ on $\Omega_h$,   we have $D\cdot \tilde{u}^n=0$ on $h\Z^3$.  
\item The form of the discrete $t$-derivative in \eqref{d1} is necessary for the CFL-condition to be fulfilled.
\item  The same form of the discrete $t$-derivative is required in \eqref{d3} for consistency in energy estimates (the energy inequality must contain the terms exactly the same as the left hand side of \eqref{d1} for cancelation). 
\item The second and third terms in the left hand side of \eqref{d3} are corresponding to $\sum_{j=1}^3\p_{x_j}(\rho v_j v)=\sum_{j=1}^3\{(\p_{x_j}(\rho v_j)v+\rho v_j(\p_{x_j}v)\}$.
\item $q^{n+1}$ is necessary to verify \eqref{d5}, where additional conditions for the mean value of $q^{n+1}$ is necessary to obtain $q^{n+1}$ uniquely. 
\item \eqref{d3} is a version of Ladyzhenskaya's discrete scheme for the homogeneous incompressible Navier-Stokes equations \cite{KL}, \cite{Ladyzhenskaya}, where it is designed so that the nonlinear term has null-contribution in $L^2$-estimates. 
\end{itemize} 
\subsection{Unique solvability}

We prove the unique solvability of our discrete problem. For this purpose, we impose  the $0$-mean value condition on  $q^{n+1}$ over $\Omega_h^{\circ}\cap G^i$ for each $i=1,2,\ldots,8$.  
\begin{Prop}\label{us}
 For given $\eta^n$ with $\rho^0_{\ast}\le \eta^n\le\rho^0_{\ast\ast}$ and $u^n$ with $u^n=0$ on $\p\Omega_h$ and $D\cdot u^n=0$ on $\Omega_h$ (if $n=0$, the conditions $u^0=0$ on $\p\Omega_h$ and $D\cdot u^0=0$ on $\Omega_h$ are not required), there exist $\eta^{n+1}$  with $\rho^0_{\ast}\le \eta^{n+1}\le\rho^0_{\ast\ast}$, $u^{n+1}$ and $q^{n+1}$ that solve \eqref{d1}-\eqref{d5}; $\eta^{n+1}$ and $u^{n+1}$ are unique, while $q^{n+1}$ is unique up to its mean value over $\Omega_h^{\circ}\cap G^i$. 
\end{Prop}
\begin{proof}
It is clear that $\eta^{n+1}$ is uniquely obtained by \eqref{d1} and \eqref{d2}. In order to check  $\rho^0_{\ast}\le \eta^{n+1}\le\rho^0_{\ast\ast}$, rewrite \eqref{d1} as 
\begin{eqnarray}\label{explicit}
\eta^{n+1}(x)&=&\frac{1}{7}\eta^{n}(x)
+\sum_{j=1}^3\Big(\frac{1}{7}+\frac{\tau}{2h}\tilde{u}^n_j(x-he^j)\Big)\eta^n(x-he^j)\\\nonumber
&&\quad +\sum_{j=1}^3\Big(\frac{1}{7}-\frac{\tau}{2h}\tilde{u}^n_j(x+he^j)\Big)\eta^n(x+he^j).
\end{eqnarray}
Due to the scale condition \eqref{scale}, the bound of $\tilde{u}^n$ and the discrete divergence-free constraint of $\tilde{u}^n$, we have 
\begin{eqnarray}\label{CFL}
&&\frac{1}{7}\pm\frac{\tau}{2h}\tilde{u}^n_j(x\mp he^j)\ge0\qquad(\mbox{the CFL-condition}),\\\nonumber
&&\frac{1}{7}+\sum_{j=1}^3\Big\{\Big(\frac{1}{7}+\frac{\tau}{2h}\tilde{u}^n_j(x-he^j)\Big)+\Big(\frac{1}{7}-\frac{\tau}{2h}\tilde{u}^n_j(x+he^j)\Big)\Big\}=1.
\end{eqnarray}
Hence, if $\eta^0_{\ast}\le\eta^n\le\eta^0_{\ast\ast}$, we have $\eta^0_{\ast}\le\eta^{n+1}(x)\le\eta^0_{\ast\ast}$. This reasoning works also for $n=0$ due to the definition of initial data.  

We discuss the unique existence of $u^{n+1}$ and $q^{n+1}$. Our argument will also show how to construct $u^{n+1}$ and $q^{n+1}$. Suppose the $0$-mean value condition of $q^{n+1}$: 
\begin{eqnarray}\label{3-mean}
\sum_{x\in\Omega_h^\circ\cap G^i}q^{n+1}(x)=0,\quad i=1,2,\ldots,8. 
\end{eqnarray}
We note that any function $w:\Omega_h\to\R^3$ with $w|_{\partial\Omega_h}=0$ satisfies 
\begin{eqnarray}\label{trivial}
\,\,\,\,\,\,\sum_{x\in\Omega_h\cap G^i}D\cdot w(x)=\sum_{j=1}^3\sum_{x\in\Omega_h\cap G^i}\frac{w_j(x+he^j)-w_j(x-he^j)}{2h}=0,\quad i=1,\cdots,8
\end{eqnarray}
due to cancelation.  We label each point of $\Omega_h\setminus\partial\Omega_h$ and $\partial\Omega_h$ as 
$$ \Omega_h\setminus\partial\Omega_h=\{x^{1},x^{2},\ldots,x^{a}  \},\quad\partial\Omega_h=\{\bar{x}^1,\bar{x}^2,\ldots,\bar{x}^b\}.$$
Set $y\in\R^{4a+b}$ and $\alpha\in\R^{4a+b+8}$ as  
\begin{eqnarray*}
y&=&\big(u^{n+1}_1(x^{1}),\ldots,u^{n+1}_1(x^{a}),u^{n+1}_2(x^{1}),\ldots,u^{n+1}_2(x^{a}),u^{n+1}_3(x^{1}),\ldots,u^{n+1}_3(x^{a}), \\
&&q^{n+1}(x^1),\ldots,q^{n+1}(x^{a}),
q^{n+1}(\bar{x}^1),\ldots,q^{n+1}(\bar{x}^b)   \big),\\
\alpha&=&\big(0,\ldots,0,b_1(x^{1}),\ldots,b_1(x^{a}),b_2(x^{1}),\ldots,b_2(x^{a}),b_3(x^{1}),\ldots,b_3(x^{a}),\\
&&\quad 0,0,0,0,0,0,0,0        \big)\mbox{\quad  with}\\
b(x)&=& \frac{1}{7}\sum_{\omega\in B}\eta^n(x+\omega h)u^n(x+\omega h)+ \eta^{n+1}(x)f^{n+1}(x)\tau,
\end{eqnarray*}
where $\alpha$ has $a+b$ zeros coming from \eqref{d5} in front of $b_1(x^1)$.   We see that the equations \eqref{d3}-\eqref{d5} and \eqref{3-mean}  are written as  a $(4a+b+8)$-system of linear equations, which is denoted by $\tilde{A}y=\alpha$ with a $(4a+b+8)\times(4a+b)$-matrix $\tilde{A}=\tilde{A}(\eta^{n+1},\eta^n,u^n,\tau,h)$.  Since \eqref{d4} and \eqref{d5} implies  \eqref{trivial}$_{w=u^{n+1}}$, we find the eight trivial equalities $0=0$ in $\tilde{A}y=\alpha$. Hence, $\tilde{A}y=\alpha$ can be deduced to be of the form $Ay=\beta$ with a $(4a+b)\times(4a+b)$-matrix $A=A(\eta^{n+1},\eta^n,u^n,\tau,h)$ and $\beta\in \R^{4a+b}$.

Our proof is complete, if $A$ is proven to be invertible, i.e., $Ay=0$ if and only if $y=0$. We have  at least one solution $y$ to $Ay=0$. Then, we obtain  at least one pair $u^{n+1},q^{n+1}$ satisfying 
\begin{eqnarray}\label{331}
&&\Big(\eta^{n+1}u^{n+1}_i(x)- 0 \Big)\frac{1}{\tau}
+D\cdot (\eta^n\tilde{u}^n)(x)u^{n+1}_i(x)\\\nonumber
&&\qquad  +\sum_{j=1}^3\frac{1}{2}\Big(
\eta^n(x-he^j)\tilde{u}^n_j(x-he^j)D_ju^{n+1}_i(x-he^j)\\\nonumber
&& \qquad +
\eta^n(x+he^j)\tilde{u}^n_j(x+he^j)D_ju^{n+1}_i(x+he^j)\Big)\\\nonumber
&&=D^-\cdot\Big\{\mu(\eta^{n+1})\Big(D^+u^{n+1}_i + D^+_iu^{n+1}\Big)\Big\}(x) -D_iq^{n+1}(x),\\\nonumber
&&\qquad\qquad \qquad \qquad  \qquad \qquad \qquad \quad\,\,\,  \mbox{ on } \Omegain,\,\,\,i=1,2,3, \\\label{332}
&&D\cdot u^{n+1}=0\quad\mbox{ on $\Omega_h$}, \quad u^{n+1}=0\quad \mbox{ on $\partial\Omega_h$},\\\label{333}
&& \sum_{x\in\Omega_h^{\circ}\cap G^j} q^{n+1}(x)=0, \quad j=1,\cdots,8
\end{eqnarray}
Due to the summation by parts, we have 
\begin{eqnarray}\label{cal-Korn}
&&\sum_{i=1}^3\sum_{x\in\Omegain}  D^-\cdot\Big\{\mu(\eta^{n+1})\Big(D^+u^{n+1}_i + D^+_iu^{n+1}\Big)\Big\}(x) u^{n+1}_i(x)\\\nonumber
&&=-\sum_{i=1}^3\sum_{x\in\Omega_h}   \mu(\eta^{n+1}(x))\Big(D^+u^{n+1}_i(x) + D^+_iu^{n+1}(x)\Big)\cdot D^+u^{n+1}_i(x)\\\nonumber
&&=-\sum_{i,j=1}^3\sum_{x\in\Omega_h}   \mu(\eta^{n+1}(x))\Big(D^+_ju^{n+1}_i(x)D^+_ju^{n+1}_i(x) + D^+_iu^{n+1}_j(x)D^+_ju^{n+1}_i(x)\Big) \\\nonumber
&&=-\2\sum_{i,j=1}^3\sum_{x\in\Omega_h}   \mu(\eta^{n+1}(x))\Big(D^+_ju^{n+1}_i(x)D^+_ju^{n+1}_i(x) + D^+_iu^{n+1}_j(x)D^+_ju^{n+1}_i(x)\Big) \\\nonumber
&&\quad -\2\sum_{i,j=1}^3\sum_{x\in\Omega_h}   \mu(\eta^{n+1}(x))\Big(D^+_iu^{n+1}_j(x)D^+_iu^{n+1}_j(x) + D^+_ju^{n+1}_i(x)D^+_iu^{n+1}_j(x)\Big) \\\nonumber
&&=-\2\sum_{i,j=1}^3\sum_{x\in\Omega_h}  
\mu(\eta^{n+1}(x))\Big(D^+_ju^{n+1}_i(x)+D^+_iu^{n+1}_j(x)\Big)^2.
\end{eqnarray}
Similarly, we have 
\begin{eqnarray*}
&&\sum_{x\in\Omegain}\!\!\!Dq^{n+1}(x)\cdot u^{n+1}(x)=\sum_{x\in\Omega_h}Dq^{n+1}(x)\cdot u^{n+1}(x)=
\sum_{x\in\Omega_h}q^{n+1}(x)D\cdot u^{n+1}(x)=0,\\
&&\sum_{x\in\Omegain}\frac{1}{2}\Big(
\eta^n(x-he^j)\tilde{u}^n_j(x-he^j)D_ju^{n+1}(x-he^j)\\\nonumber
&&\qquad\qquad +
\eta^n(x+he^j)\tilde{u}^n_j(x+he^j)D_ju^{n+1}(x+he^j)\Big)\cdot u^{n+1}(x)\\
&&\quad= \frac{1}{2h}\sum_{x\in\Omega_h}\Big(\eta^n(x-he^j)\tilde{u}^n_j(x-he^j)-\eta^n(x+he^j)\tilde{u}^n_j(x+he^j)\Big)   |u^{n+1}(x)|^2\\
&&\qquad - \underline{\frac{1}{2h}\sum_{x\in\Omega_h}
\eta^n(x-he^j)\tilde{u}^n_j(x-he^j)u^{n+1}(x-2he^j)\cdot u^{n+1}(x)}_{\rm(i)}\\
&&\qquad + \underline{\frac{1}{2h}\sum_{x\in\Omega_h}
\eta^n(x+he^j)\tilde{u}^n_j(x+he^j)u^{n+1}(x+2he^j)\cdot u^{n+1}(x)}_{\rm(ii)}\\
&&\quad=- \frac{1}{2}\sum_{x\in\Omega_h}D_j(\eta^n\tilde{u}^n_j)(x) |u^{n+1}(x)|^2,
\end{eqnarray*}
where we see that (i)$=$(ii) by shifting $x$ to $x\pm he^j$ in (i), (ii), respectively.   Hence, we obtain by  \eqref{331}$ \times u^{n+1}_i$ with \eqref{332}, \eqref{333} and \eqref{d1},
\begin{eqnarray*}
0&=&\frac{1}{\tau}\sum_{x\in\Omega_h}\eta^{n+1}(x)|u^{n+1}(x)|^2
+\frac{1}{2}\sum_{x\in\Omega_h}D\cdot(\eta^n\tilde{u}^n)(x) |u^{n+1}(x)|^2\\
&&+\2\sum_{i,j=1}^3\sum_{x\in\Omega_h}  
\mu(\eta^{n+1}(x))\Big(D^+_ju^{n+1}_i(x)+D^+_iu^{n+1}_j(x)\Big)^2\\
&=&\frac{1}{\tau}\sum_{x\in\Omega_h}\eta^{n+1}(x)|u^{n+1}(x)|^2
-\frac{1}{2\tau}\sum_{x\in\Omega_h}\Big(\eta^{n+1}(x)-\frac{1}{7}\sum_{\omega\in B}\eta^n(x+\omega h)    \Big) |u^{n+1}(x)|^2\\
&&+\2\sum_{i,j=1}^3\sum_{x\in\Omega_h}  
\mu(\eta^{n+1}(x))\Big(D^+_ju^{n+1}_i(x)+D^+_iu^{n+1}_j(x)\Big)^2\\
&=&\frac{1}{2\tau}\sum_{x\in\Omega_h}\eta^{n+1}(x)|u^{n+1}(x)|^2+\frac{1}{14\tau}\sum_{x\in\Omega_h}\sum_{\omega\in B}\eta^n(x+\omega h) |u^{n+1}(x)|^2\\
&&+\2\sum_{i,j=1}^3\sum_{x\in\Omega_h}  
\mu(\eta^{n+1}(x))\Big(D^+_ju^{n+1}_i(x)+D^+_iu^{n+1}_j(x)\Big)^2, 
\end{eqnarray*}
which leads to $u^{n+1}=0$ on $\Omega_h$ due to the positivity of $\eta^n$, $\eta^{n+1}$ and $\mu(\cdot)$.   Therefore, \eqref{331} implies $Dq^{n+1}=0$ on $\Omegain$, i.e., $q^{n+1}$ is constant on $\Omega_h\cap G^i$. \eqref{3-mean} yields $q^{n+1}=0$ on $\Omega_h$. Thus, we conclude that $Ay=0$ only admits the trivial solution and $A$ is invertible. This reasoning works for $n=0$ as well.  
\end{proof}
\subsection{A priori estimates}
We provide $(\tau,h)$-independent estimates for the discrete problems that are required in the convergence proofs given in Section 4. Note that we do not necessarily  seek for the sharpest estimates. 
\begin{Prop}\label{3estimate}
The solution  of the discrete problem \eqref{d1}-\eqref{d5} satisfies for all $1\le n+1\le T_\tau$,
\begin{eqnarray}\label{e-estimate0}
&&\rho^0_{\ast}\le \eta^{n+1}\le \rho^0_{\ast\ast},\quad 0<\mu_\ast\le \mu(\eta^{n+1})\le \mu_{\ast\ast}\quad(\mbox{$\mu_\ast$, $\mu_{\ast\ast}$ are some constants)}, 
\\\label{e-estimate1}
&&\norm \eta^{n+1}\norm_{p,\tOmega_h}\le  \norm \rho^0\norm_{L^p(\tOmega)}-\sum_{m=0}^n\sum_{x\in\tOmegain} \!\!\!\!\! D\cdot(|\eta^m|^p\tilde{u}^m )(x)\tau,\,\,\,\forall\,p\in[1,\infty),
\\\label{e-estimate2}
&& \norm \sqrt{\eta^{n+1}}u^{n+1}\norm_{2,\Omega_h}^2\le
\norm \sqrt{\eta^{n}}u^{n}\norm_{2,\Omega_h}^2-2\mu_\ast\sum_{j=1}^3\norm  D_j^+u^{n+1}\norm_{2,\Omega_h}^2\tau \\\nonumber 
&&\qquad\qquad\qquad\qquad\quad +2(\sqrt{\eta^{n+1}}f^{n+1},\sqrt{\eta^{n+1}}u^{n+1})_{\Omega_h}\tau.
\end{eqnarray}
\end{Prop}
\begin{proof}
We already proved the first inequality of \eqref{e-estimate0} in the proof of Proposition \ref{us}; the second one in  \eqref{e-estimate0} follows from the positivity of $\mu|_{[\rho^0_\ast,\rho^0_{\ast\ast}]}$.  
Let $p^\ast$ be the H\"older conjugate of $p\in(1,\infty)$. Observe that 
\begin{eqnarray*}
|\eta^0(x)|^{p}&\le&\Big(\frac{1}{h^3}\int_{C_h(x)}|\rho^0(y)|dy\Big)^{p}
\le\Big\{\frac{1}{h^3}\Big(\int_{C^+_h(x)}|\rho^0(y)|^pdy\Big)^\frac{1}{p}\Big(\int_{C^+_h(x)}1^{p^\ast}dy\Big)^\frac{1}{p^\ast}\Big\}^{p}\\
&=&\frac{1}{h^3}\int_{C^+_h(x)}|\rho^0(y)|^pdy \quad (\mbox{the case of $p=1$ is also clear}),\\
\norm \eta^0\norm_{p,\tOmega_h}&\le&\norm\rho^0\norm_{L^p(\tOmega)}\le{\rm vol}(\tOmega)^\frac{1}{p}\norm \rho^0\norm_{L^\infty(\Omega)},\,\,\,\,\forall\,p\in[1,\infty).
\end{eqnarray*}
Rewrite \eqref{explicit} as 
$$g^{n+1}(x) = \sum_{\omega\in B} g^n(x+h \omega)\nu(\omega),$$
where
$$\nu(0)=\frac{1}{7},\,\,\,\nu(-e^j)=\frac{1}{7}+\frac{\tau}{2h}\tilde{u}^n_j(x-he^j),\,\,\,\nu(e^j)=\frac{1}{7}-\frac{\tau}{2h}\tilde{u}^n_j(x+he^j).$$
Note that $\nu(\omega)\ge0$ due to \eqref{CFL} and $\sum_{\omega \in B}\nu(\omega)=1$.
For each $x\in \tOmegain$, it holds that  
$$|\eta^{n+1}(x)| \le \sum_{\omega\in B} |\eta^n(x+h \omega)|\nu(\omega).$$
Applying the (discrete) H\"older inequality to the right hand side with respect to $\omega$, we obtain 
\begin{eqnarray*}
&&|\eta^{n+1}(x)| \le \sum_{\omega\in B} |\eta^n(x+h \omega)|\nu(\omega)\le \Big(\sum_{\omega\in B}|\eta^n(x+h \omega)|^p\nu(\omega)\Big)^\frac{1}{p}\Big(\sum_{\omega\in B}1^{p^\ast}\nu(\omega)\Big)^\frac{1}{p^\ast},\\
&&|\eta^{n+1}(x)|^p\le \sum_{\omega\in B}|\eta^n(x+h \omega)|^p\nu(\omega),\quad\forall\,p\in(1,\infty),
\end{eqnarray*}
which leads to 
\begin{eqnarray}\label{explicit1-2}
&& |\eta^{n+1}(x)|^p \le  
 \frac{1}{7}\sum_{\omega\in B}|\eta^n(x+h \omega)|^p   -D\cdot( |\eta^n|^p\tilde{u}^n)(x)\tau
,\,\,\,\forall\,p\in[1,\infty),\\\nonumber
&& D\cdot(|\eta^n|^p\tilde{u}^n)(x)=\sum_{j=1}^3\frac{|\eta^n(x+he^j)|^p\tilde{u}^n_j(x+he^j)-|\eta^n(x-he^j)|^p\tilde{u}^n_j(x-he^j)}{2h}.
\end{eqnarray}
 Noting  $\eta^n|_{\p \tOmega_h}=\eta^{n+1}|_{\p\tOmega_h}=\rho^0_{\ast}$ and $\rho^0_{\ast}\le\eta^{n+1}$,   we sum up  \eqref{explicit1-2}$\times h^3$ over $x\in \tOmegain$ to obtain  
 \begin{eqnarray*}
&&\norm \eta^{n+1}\norm_{p,\tOmegain}^p\le \norm \eta^n\norm_{p,\tOmegain}^p-\sum_{x\in\tOmegain} D\cdot( |\eta^n|^p\tilde{u}^n)(x)\tau,\quad\forall\,p\in[1,\infty),\\
&&\norm \eta^{n+1}\norm_{p,\tOmega_h}^p\le \norm \eta^n\norm_{p,\tOmega_h}^p-\sum_{x\in\tOmegain} D\cdot( |\eta^n|^p\tilde{u}^n)(x)\tau,\quad\forall\,p\in[1,\infty),
\end{eqnarray*}
which yields \eqref{e-estimate1} for $p\in[1,\infty)$.

We prove \eqref{e-estimate2}. Since $\eta^{n+1}$ and $\eta^n$ are non-negative, it follows from the inequality of arithmetic and geometric means that 
\begin{eqnarray*} 
&&\sum_{x\in\Omegain}\Big(\eta^{n+1}(x)u^{n+1}(x)-\frac{1}{7}\sum_{\omega\in B}\eta^n(x+h\omega)u^n(x+h\omega)\Big)\cdot u^{n+1}(x)h^3\\
&&\quad \ge \sum_{x\in\Omegain} \Big(\eta^{n+1}(x)|u^{n+1}(x)|^2- \frac{1}{7}\sum_{\omega\in B}\eta^n(x+h\omega)\frac{|u^n(x+h\omega)|^2+ |u^{n+1}(x)|^2}{2}\Big)h^3\\
&&\quad \ge \frac{1}{2}\sum_{x\in\Omega_h} \eta^{n+1}(x)|u^{n+1}(x)|^2h^3
-\frac{1}{2}\sum_{x\in\Omega_h} \eta^{n}(x)|u^{n}(x)|^2h^3\\
&&\qquad +\frac{1}{2}\sum_{x\in\Omegain} \Big(\eta^{n+1}(x)-\frac{1}{7}\sum_{\omega\in B}\eta^n(x+h\omega)\Big) |u^{n+1}(x)|^2h^3.
\end{eqnarray*}
 By calculations done in the proof of Proposition \ref{us}, we have 
\begin{eqnarray*}
&&\sum_{x\in\Omegain}\Big\{\sum_{j=1}^3D_j(\eta^n\tilde{u}^n_j)(x)u^{n+1}(x) +\sum_{j=1}^3\frac{1}{2}\Big(
\eta^n(x-he^j)\tilde{u}^n_j(x-he^j)D_ju^{n+1}(x-he^j)\\
&&\qquad +
\eta^n(x+he^j)\tilde{u}^n_j(x+he^j)D_ju^{n+1}(x+he^j)\Big)\Big\}\cdot u^{n+1}(x)h^3\\
&&\quad =\frac{1}{2}\sum_{x\in\Omegain} D\cdot(\eta^n\tilde{u}^n)(x)|u^{n+1}(x)|^2h^3,\\
&&\sum_{x\in\Omegain}\!\!\!Dq^{n+1}(x)\cdot u^{n+1}(x)h^3=
\sum_{x\in\Omega_h}q^{n+1}(x)(D\cdot u^{n+1})(x)h^3=0.
\end{eqnarray*}
By the calculation \eqref{cal-Korn} and Lemma \ref{Korn}, we have 
\begin{eqnarray*}
&&\sum_{i=1}^3\sum_{x\in\Omegain}  D^-\cdot\Big\{\mu(\eta^{n+1})\Big(D^+u^{n+1}_i+D^+_iu^{n+1}\Big)  \Big\}(x) u^{n+1}_i(x)
h^3\\
&&=-\2\sum_{i,j=1}^3\sum_{x\in\Omega_h}  \mu(\eta^{n+1}(x))\Big(D^+_ju^{n+1}_i(x)+D^+_iu^{n+1}_j(x)    \Big)^2h^3\\
&&\le -\mu_\ast\sum_{j=1}^3\sum_{x\in\Omega_h} |D^+_ju^{n+1}(x)|^2h^3.
\end{eqnarray*}
Hence, \eqref{d3}$\times u^{n+1}_i$ yields 
\begin{eqnarray*}
&& \frac{1}{2}\sum_{x\in\Omega_h} \eta^{n+1}(x)|u^{n+1}(x)|^2h^3
-\frac{1}{2}\sum_{x\in\Omega_h} \eta^{n}(x)|u^{n}(x)|^2h^3\\
&&\qquad +\frac{1}{2}\sum_{x\in\Omegain}\Big\{ \Big(\eta^{n+1}(x)-\frac{1}{7}\sum_{\omega\in B}\eta^n(x+h\omega)\Big)- D\cdot(\eta^n\tilde{u}^n)(x)\tau\Big\}  |u^{n+1}(x)|^2h^3\\
&&\quad \le  -\mu_\ast\sum_{j=1}^3\sum_{x\in\Omega_h} |D^+_ju^{n+1}(x)|^2h^3\tau+\sum_{x\in\Omega_h}\eta^{n+1}(x)f^{n+1}(x)u^{n+1}(x)h^3\tau, 
\end{eqnarray*}
where the term $\{\cdot\}$ is equal to $0$ due to \eqref{d1}. 
\end{proof}
\begin{Cor}\label{3es}
 The solution  of the discrete problem \eqref{d1}-\eqref{d5} satisfies for all $1\le n+1\le T_\tau$,
 \begin{eqnarray}\label{3es1}
&&\rho^0_{\ast}\norm u^{n+1}\norm_{2,\Omega_h}^2 \le \norm\sqrt{\eta^{n+1}}u^{n+1}\norm_{2,\Omega_h}^2  \\\nonumber
&&\qquad \le (1+2e^{2T+2})\rho^0_{\ast\ast}\Big(\norm v^0\norm_{L^2(\Omega)^3}^2+\norm f\norm_{L^2([0,T+1];L^2(\Omega)^3)}^2
 \Big),\\\label{3es2}
 && 2\mu_\ast\sum_{m=1}^{n+1}\sum_{j=1}^3\norm D_j^+u^m\norm_{2,\Omega_h}^2\tau\\\nonumber
&&\qquad \le\{1  +(1+2 e^{2T+2})(T+1)   \}\rho^0_{\ast\ast}(\norm v^0\norm_{L^2(\Omega)^3}^2+\norm f\norm_{L^2([0,T+1];L^2(\Omega)^3)}^2).
\end{eqnarray}
\end{Cor}
 \begin{proof}
 It follows from  \eqref{e-estimate1} and \eqref{e-estimate2} that for any $1\le n+1\le T_\tau$,
 \begin{eqnarray*}
 \norm\sqrt{\eta^{n+1}}u^{n+1}\norm_{2,\Omega_h}^2 &\le& \norm\sqrt{\eta^{0}}u^{0}\norm_{2,\Omega_h}^2 + 2\sum_{m=1}^{n+1}(\sqrt{\eta^m}f^m,\sqrt{\eta^m}u^m)_{\Omega_h}\tau\\
&\le&  \norm\sqrt{\eta^{0}}u^{0}\norm_{2,\Omega_h}^2+2\sum_{m=1}^{n+1}\norm\sqrt{\eta^m}f^m\norm_{2,\Omega_h}\norm\sqrt{\eta^m}u^m\norm_{2,\Omega_h}\tau\\
 &\le& \norm\sqrt{\eta^{0}}u^{0}\norm_{2,\Omega_h}^2+ \sum_{m=1}^{n+1}\norm\sqrt{\eta^m}f^m\norm_{2,\Omega_h}^2\tau+\sum_{m=1}^{n+1}\norm\sqrt{\eta^m}u^m\norm_{2,\Omega_h}^2\tau,\\
 \sum_{m=1}^{n+1}\norm\sqrt{\eta^m}f^m\norm_{2,\Omega_h}^2\tau&\le& \rho^0_{\ast\ast}\norm f\norm_{L^2([0,T+1];L^2(\Omega)^3)}^2. 
 \end{eqnarray*}
 Set $X^m:=\sum_{m'=1}^{m}\norm\sqrt{\eta^{m'}}u^{m'}\norm_{2,\Omega_h}^2\tau$ for $m\in\N$ and $a:= \rho^0_{\ast\ast}(\norm v^{0}\norm_{L^2(\Omega)^3}^2+\norm f\norm_{L^2([0,T+1];L^2(\Omega)^3)}^2)$. Then, we have
$$\frac{X^{n+1}-X^n}{\tau}\le a+X^{n+1},$$
from which we obtain for any $0<\tau\le\frac{1}{2}$,
\begin{eqnarray*}
&&X^{n+1}\le\frac{1}{1-\tau}X^n+\frac{\tau a}{1-\tau}\le (1+2\tau)X^n+(1+2\tau)\tau a,\\
&&  \Big(X^{n+1}+\frac{1+2\tau}{2}a\Big)\le (1+2\tau)\Big(X^{n}+\frac{1+2\tau}{2}a\Big).
\end{eqnarray*}
Hence, we have for any $2\le 1+n\le T_\tau$,
$$X^{n+1}\le (1+2\tau)^{n}\Big(X^{1}+\frac{1+2\tau}{2}a\Big)\le e^{2T+2}(X^{1}+a).$$
We can directly estimate $X^1$ through  \eqref{e-estimate2}$_{n=0}$ as 
\begin{eqnarray*}
X^1&\le&\Big(\frac{\tau}{1-\tau} \norm \sqrt{\eta^0}u^0\norm_{2,\Omega_h}^2+\frac{\tau^2}{1-\tau}\norm \sqrt{\eta^1}f^1\norm_{2,\Omega_h}^2 \Big)\\
&\le& \rho^0_{\ast\ast}(\norm v^0\norm_{L^2(\Omega)^3}^2+\tau\norm f\norm_{L^2([0,\tau];L^2(\Omega)^3)}^2)\\
&\le& \rho^0_{\ast\ast}(\norm v^0\norm_{L^2(\Omega)^3}^2+\norm f\norm_{L^2([0,T+1];L^2(\Omega)^3)}^2).
\end{eqnarray*}
Therefore, we conclude that for any $2\le n+1\le T_\tau$,
\begin{eqnarray*}
&& \norm\sqrt{\eta^{n+1}}u^{n+1}\norm_{2,\Omega_h}^2 \le
 a+X^{n+1}\\
 &&\quad \le (1+2 e^{2T+2})\rho^0_{\ast\ast}\Big(\norm v^0\norm_{L^2(\Omega)^3}^2+\norm f\norm_{L^2([0,T];L^2(\Omega)^3)}^2
 \Big).
\end{eqnarray*}
Through  \eqref{e-estimate2}$_{n=0}$, we see that 
$$ \norm\sqrt{\eta^{1}}u^{1}\norm_{2,\Omega_h}^2 \le2\rho^0_{\ast\ast}(\norm v^0\norm_{L^2(\Omega)^3}^2+\norm f\norm_{L^2([0,T+1];L^2(\Omega)^3)}^2).
$$
\indent It follows from \eqref{e-estimate1} and \eqref{e-estimate2} that for any $1\le n+1\le T_\tau$,
\begin{eqnarray*}
&&2\mu_\ast\sum_{m=1}^{n+1}\sum_{j=1}^3\norm D_j^+u^m\norm_{2,\Omega_h}^2\tau\le \norm \sqrt{\eta^0}u^0\norm_{2,\Omega_h}^2+\sum_{m=1}^{n+1}\norm\sqrt{\eta^m}f^{m}\norm_{2,\Omega_h}^2\tau\\
&&\qquad +\sum_{m=1}^{n+1}\norm\sqrt{\eta^{m}}u^{m}\norm_{2,\Omega_h}^2\tau\\
&&\quad\le \{1  +(1+2 e^{2T+2})(T+1)   \}\rho^0_{\ast\ast}(\norm v^0\norm_{L^2(\Omega)^3}^2+\norm f\norm_{L^2([0,T+1];L^2(\Omega)^3)}^2).
\end{eqnarray*}
 \end{proof}
Although convergence of $q^{n+1}$ is not required, we need some estimates for it in Section 4. Taking the inner product of \eqref{d3}$_{i=1,2,3}$ and $Dq^{n+1}$, we have 
\begin{eqnarray*}
&&\norm Dq^{n+1}\norm_{2,\Omegain}^2\le (\norm  \eta^{n+1}u^{n+1}\norm_{2,\Omega_h}+\norm  \eta^{n}u^{n}\norm_{2,\Omega_h})\tau^{-1}\norm Dq^{n+1}\norm_{2,\Omegain}\\
&&\quad +\sum_{j=1}^3\norm \eta^n\tilde{u}^n_j\norm_{\infty,\Omega_h}h^{-1}\norm u^{n+1}\norm_{2,\Omega_h} \norm Dq^{n+1}\norm_{2,\Omegain}  \\\nonumber
&&\quad  +\sum_{j=1}^3
\norm \eta^n\tilde{u}^n_j\norm_{\infty,\Omega_h}\norm D_ju^{n+1}\norm_{2,\Omega_h}\norm Dq^{n+1}\norm_{2,\Omegain}\\\nonumber
&&\quad +\mu_{\ast\ast}\sum_{j=1}^32\Big(\norm D_j^+u^{n+1}\norm_{2,\Omega_h}+\norm D^+u^{n+1}_j\norm_{2,\Omega_h}\Big)h^{-1}\norm Dq^{n+1}\norm_{2,\Omegain}\\
&&\quad +\norm \eta^{n+1}f^{n+1}\norm_{2,\Omega_h}\norm Dq^{n+1}\norm_{2,\Omegain}.
\end{eqnarray*}
By the already obtained estimates and the bound of $\tilde{u}^n$, we obtain with a $(\tau,h)$-independent constant $M>0$, 
\begin{eqnarray}\label{qqq}
 \norm Dq^{n+1}\norm_{2,\Omegain}&\le& M\tau^{-1}+Mh^{-2+\alpha}+ Mh^{-1}\sum_{j=1}^3\norm D_j^+u^{n+1}\norm_{2,\Omega_h}\\\nonumber
&& +M\norm f^{n+1}\norm_{2,\Omega_h},\quad1\le \forall\,n+1\le T_\tau.
\end{eqnarray}
This estimate will be used in the following way: for $\phi\in C^3_{0,\sigma}(\Omega)$, where $\phi|_{\Omega_h}$ is still denoted by $\phi$ and supp$(\phi)\cap \Omega_h\subset \Omega_h^{\circ}$ for all sufficiently small $h>0$, it holds that
\begin{eqnarray*}
|(Dq^{n+1},\phi)_{\Omega_h}|=|(q^{n+1},D\cdot\phi)_{\Omega_h}|\le O(h^2)\norm q^{n+1}\norm_{\Omega_h^{\circ}};
\end{eqnarray*}
due to  Lemma \ref{PoinII}, we have  
\begin{eqnarray}\label{qqqq}
\quad |(Dq^{n+1},\phi)_{\Omega_h}|&\le& O(h^2)\norm Dq^{n+1}\norm_{\Omegain}\\\nonumber
&=& O(h^{\alpha})+O(h) \sum_{j=1}^3\norm D_j^+u^{n+1}\norm_{2,\Omega_h}+O(h^2)\norm f^{n+1}\norm_{2,\Omega_h}.
\end{eqnarray}
\setcounter{section}{3}
\setcounter{equation}{0}
\section{Convergence}
We investigate weak and strong convergence of the solution to the discrete problem. 
For each $\delta:=(h,\tau)$, define the step functions  $\rho_\delta:[0,T]\times\tOmega\to\R$, $v_\delta,w^i_\delta:[0,T]\times\Omega\to\R^3$, $i=1,2,3$ generated by the solution of \eqref{d1}-\eqref{d5}: on $(0,T]$,    
\begin{eqnarray*}
\rho_\delta(t,x)&:=&\left\{
\begin{array}{lll}
&\eta^{n +1}(y)\mbox{\quad\quad\ for $t\in(n\tau,n\tau+\tau]$, $x\in {C^+_h(y)}$, $y\in \tOmega_h$},
\medskip\\
&\rho^0_\ast\mbox{\quad\quad\quad\,\,\,\quad\, otherwise},
\end{array}
\right. \\
v_\delta(t,x)&:=&\left\{
\begin{array}{lll}
&u^{n+1}(y)\mbox{\quad\quad\ for $t\in(n\tau,n\tau+\tau]$, $x\in {C_h^+(y)}$, $y\in \Omega_h$},
\medskip\\
&0\mbox{\quad\quad\quad\,\,\,\quad\,\,\, otherwise},
\end{array}
\right. \\
w^i_\delta(t,x)&:=&\left\{
\begin{array}{lll}
&D_i^+u^{n+1}(y)\mbox{\quad for $t\in(n\tau,n\tau+\tau]$, $x\in {C_h^+(y)}$, $y\in \Omega_h$},
\medskip\\
&0\mbox{\quad\quad\quad\quad\quad\, otherwise}, 
\end{array}
\right.
\end{eqnarray*}
where $n=0,1,\ldots,T_\tau-1$, 
and for $t=0$ the value of each step function is defined as the value at $t=\tau$.  
 In the rest of our argument, the statement ``there exists a sequence $\delta\to0$ ...'' means ``there exists a sequence $\delta_l=(h_l,\tau_l)$ with $h_l,\tau_l\searrow0$ as $l\to\infty$ ...''. 
\subsection{Weak convergence} 

We first investigate weak convergence, which is rather straightforward from the results in Subsection 3.2. 
The proof requires Lipschitz interpolation of functions defined on $\Omega_h$: 
\begin{Lemma}[Appendix (1) of \cite{Kuroki-Soga}]\label{inter}
For a function $u:\Omega_h\to\R$ with $u|_{\partial\Omega_h}=0$ and the step function $v$ defined as 
\begin{eqnarray*}
v(x)&:=&\left\{
\begin{array}{lll}
&u(y)\mbox{\quad for $x\in {C_{h}^+(y)}$, $y\in \Omega_{h}$},
\medskip\\
&0\mbox{\quad \,\,\,\,\,\,\,\,\mbox{otherwise}}, 
\end{array}
\right. \\
\end{eqnarray*}
there exists  a Lipschitz continuous function $w:\Omega \to \R$ with supp$(w)\subset\Omega$   such that 
\begin{eqnarray*}
&&\norm w-v\norm_{L^2(\Omega)}\le K h\norm D^+ u\norm_{\Omega_h},\\
&&\norm \partial_{x_i}w(x)\norm_{L^2(\Omega)}\le \tilde{K}\norm D^+ u\norm_{\Omega_h}, \mbox{ $i=1,2,3$},
\end{eqnarray*}
where $K$ and $\tilde{K}$ are constants independent of $u$ and $h$. 
\end{Lemma}
\begin{Prop}\label{weak-convergence}
There exists a sequence $\delta\to0$ and functions $\rho\in L^2([0,T];L^2(\tOmega))$, $v\in L^2([0,T];H^1_{0,\sigma}(\Omega))$, $\tilde{v}\in L^2([0,T];L^2(\Omega)^3)$ for which  the following weak convergence holds: 
\begin{eqnarray}\label{511}
&&\rho_\delta \wto \rho \mbox{ \quad in $L^2([0,T];L^2(\tOmega))$ as $\delta\to0$},\\\label{512}
&&v_\delta \wto v \mbox{ \quad in $L^2([0,T];L^2(\Omega)^3)$ as $\delta\to0$},\\\label{523}
&&\rho_\delta v_\delta \wto \tilde{v} \mbox{ \quad in $L^2([0,T];L^2(\Omega)^3)$ as $\delta\to0$},\\\label{523-2}
&&w^i_\delta \wto \partial_{x_i} v \mbox{ \quad in $L^2([0,T];L^2(\Omega)^3)$ as $\delta\to0$ ($i=1,2,3$)}. 
\end{eqnarray} 
\end{Prop}
\begin{proof}
Subsection 3.2 shows that  $\{\rho_\delta{}\}$, $\{v_\delta{}_j\}$, $\{\rho_\delta v_\delta{}_j\}$, $\{w^i_\delta{}_j\}$ ($i,j=1,2,3$) are bounded in the Hilbert space $L^2([0,T];L^2(\Omega))$ or $L^2([0,T];L^2(\tOmega))$. Hence, there exists a sequence $\delta\to0$ and functions  $\rho\in L^2([0,T];L^2(\tOmega))$ and  $v=(v_1,v_2,v_3),\tilde{v}=(\tilde{v}_1,\tilde{v}_2,\tilde{v}_3),w^i=(w_1^i,w_2^i,w_3^i)\in L^2([0,T];L^2(\Omega)^3)$ such that for $i,j=1,2,3$, 
\begin{eqnarray*}
&&\rho_\delta{} \wto \rho\mbox{\quad in $L^2([0,T];L^2(\tOmega))$ as $\delta\to0$},\\
&&v_\delta{}_j \wto v_j,\quad \rho_\delta v_\delta{}_j \wto \tilde{v}_j,\quad   w^i_\delta{}_j \wto w^i_j \mbox{\quad in $L^2([0,T];L^2(\Omega))$ as $\delta\to0$}.
\end{eqnarray*}
\indent In the rest of the proof, $\phi$ is such that $\phi\in C^\infty([0,T]\times\Omega)$ with supp$(\phi)\subset(0,T)\times\Omega$. Set $\phi^n(\cdot):=\phi(\tau n,\cdot)$. 

\indent We prove $\partial_{x_i} v=w^i$.  Noting the regularity of $\phi$, we have for each $n\in\N$, 
\begin{eqnarray*}
\sum_{y\in\Omega_h}D_i^+u^{n+1}_j(y)\phi^{n+1}(y) h^3
&=&-\sum_{y\in\Omega_h}u^{n+1}_j(y+he^i)D_i^+\phi^{n+1}(y) h^3\\
&=&-\sum_{y\in\Omega_h}u^{n+1}_j(y)D_i^+\phi^{n+1}(y-he^i) h^3\\
&=&-\sum_{y\in\Omega_h}u^{n+1}_j(y)D_i^+\phi^{n+1}(y) h^3+O(h),\\
(w^i_\delta{}_j,\phi)_{L^2([0,T];L^2(\Omega))}  &=& \sum_{n=0}^{T_\tau-1}\sum_{y\in\Omega_h}D_i^+u^{n+1}_j(y)(\phi^{n+1}(y)+O(\tau)+O(h)) h^3\tau\\
&=&\sum_{n=0}^{T_\tau-1}\sum_{y\in\Omega_h}D_i^+u^{n+1}_j(y)\phi^{n+1}(y)h^3\tau+O(\tau)+O(h)\\
&=&-\sum_{n=0}^{T_\tau-1}\sum_{y\in\Omega_h} u^{n+1}_j(y)D_i^+\phi^{n+1}(y) h^3\tau+O(\tau)+O(h),\\
(v_{\delta}{}_j,\partial_{x_i}\phi)_{L^2([0,T];L^2(\Omega))}  &=& \sum_{n=0}^{T_\tau-1}\sum_{y\in\Omega_h}u^{n+1}_j(y)(D_i^+\phi^{n+1}(y)+O(\tau)+O(h)) h^3\tau\\
&=&\sum_{n=0}^{T_\tau-1}\sum_{y\in\Omega_h}u^{n+1}_j(y)D_i^+\phi^{n+1}(y) h^3\tau+O(\tau)+O(h).  
\end{eqnarray*} 
Therefore, the weak convergence implies $(v_j,\partial_{x_i}\phi)_{L^2([0,T];L^2(\Omega))}=-(w^i_j,\phi)_{L^2([0,T];L^2(\Omega))}$ for any $\phi$.  

We prove $\nabla\cdot v=0$ a.e. $(t,x)\in[0,T]\times\Omega$. For each $\phi$, we have 
\begin{eqnarray*}
0&=&\sum_{n=0}^{T_\tau-1}\sum_{y\in\Omega_h}D\cdot u^{n+1}(y)\phi^{n+1}(y) h^3\tau=-\sum_{n=0}^{T_\tau-1}\sum_{y\in\Omega_h}  u^{n+1}(y) \cdot D\phi^{n+1}(y)  h^3\tau\\
&=&-\sum_{i=1}^3 (v_{\delta}{}_i,\partial_{x_i}\phi)_{L^2([0,T];L^2(\Omega))} +O(\tau)+O(h)\\
&\to&- \sum_{i=1}^3(v_i,\partial_{x_i}\phi)_{L^2([0,T];L^2(\Omega))}\mbox{\quad as $\delta\to0$}.
\end{eqnarray*} 
Therefore, we obtain $(\nabla\cdot v,\phi )_{L^2([0,T];L^2(\Omega))}=- \sum_{i=1}^3(v_i,\partial_{x_i}\phi)_{L^2([0,T];L^2(\Omega))}=0$ for any $\phi$. Up to now, we proved $v\in L^2([0,T];H^1(\Omega)^3)$ and $\nabla\cdot v=0$  a.e. $(t,x)\in[0,T]\times\Omega$. 

We prove $v\in L^2([0,T];H^1_0(\Omega)^3)$. Let $\bar{v}^{n+1}_\delta:\Omega\to\R^3$ be the Lipschitz interpolation of $u^{n+1}$ by means of Lemma \ref{inter} and let $\bar{v}_\delta:[0,T]\times\Omega\to\R^3$ be defined as $\bar{v}_\delta(t,\cdot):=\bar{v}_\delta^{n+1}(\cdot)$ for $t\in(\tau n,\tau n+\tau]\cap[0,T]$, $n=0,1,\ldots, T_\tau-1$ ($\bar{v}_\delta(0,\cdot):=\bar{v}_\delta(\tau,\cdot)$). 
Note that 
$$\bar{v}_\delta\in L^2([0,T];H^1_0(\Omega)^3),\,\,\,\norm \bar{v}_\delta-v_\delta\norm_{L^2([0,T];L^2(\Omega)^3)}=O(h),\,\,\, \norm \partial_{x_i}\bar{v}_\delta\norm_{L^2([0,T];L^2(\Omega)^3)}\le K'$$
for $i=1,2,3$, where $K'$ is a constant independent from $\delta$. We see that, taking a subsequence if necessary,  $\bar{v}_\delta\wto v$ in $L^2([0,T];L^2(\Omega)^3)$  as $\delta\to0$ and that there exists $\tilde{w}^i\in L^2([0,T];L^2(\Omega)^3)$ such that $\partial_{x_i}\bar{v}_\delta\wto \tilde{w}^i$ in $L^2([0,T];L^2(\Omega)^3)$  as $\delta\to0$ for $i=1,2,3$. Since 
$(\partial_{x_i}\bar{v}_{\delta j},\phi)_{L^2([0,T];L^2(\Omega))}=-(\bar{v}_{\delta j},\partial_{x_i}\phi)_{L^2([0,T];L^2(\Omega))}$ for any $\phi$, we have 
$  (\tilde{w}^i_j,\phi)_{L^2([0,T];L^2(\Omega))}=-(v_j,\partial_{x_i}\phi)_{L^2([0,T];L^2(\Omega))}$ and $\tilde{w}^i=\partial_{x_i}v$. In particular, 
$$(\bar{v}_\delta,\psi)_{L^2([0,T];H^1(\Omega)^3)}\to(v,\psi)_{L^2([0,T];H^1(\Omega)^3)} \mbox{ as $\delta\to0$, \,\,\, $\forall\,\psi\in L^2([0,T];H^1(\Omega)^3)$.}$$
Since $\{\bar{v}_\delta\}$ is a bounded sequence of the Hilbert space $L^2([0,T];H^1_0(\Omega)^3)$, taking a subsequence if necessary, we find $\bar{v}\in  L^2([0,T];H^1_0(\Omega)^3)$ to which $\bar{v}_\delta$ weakly converges in $L^2([0,T];H^1_0(\Omega)^3)$ as $\delta\to0$, i.e., 
$$(\bar{v}_\delta,\psi)_{L^2([0,T];H^1(\Omega)^3)}\to(\bar{v},\psi)_{L^2([0,T];H^1(\Omega)^3)} \mbox{ as $\delta\to0$, \,\,\, $\forall\,\psi\in L^2([0,T];H^1_0(\Omega)^3)$.}$$  
Therefore, we have 
$(v-\bar{v},\psi)_{L^2([0,T];H^1(\Omega)^3)}=0\mbox{ \,\,\,for any $\psi\in L^2([0,T];H^1_0(\Omega)^3)$.}$ 
Since $\bar{v}_\delta-\bar{v}\in  L^2([0,T];H^1_0(\Omega)^3)$, we obtain 
\begin{eqnarray*}
0&=&(v-\bar{v},\bar{v}_\delta-\bar{v})_{L^2([0,T];H^1(\Omega)^3)}\\
&=&(v-\bar{v},v-\bar{v})_{L^2([0,T];H^1(\Omega)^3)}+(v-\bar{v},\bar{v}_\delta-v)_{L^2([0,T];H^1(\Omega)^3)}\\
&\to&\norm v-\bar{v}\norm_{L^2([0,T];H^1(\Omega)^3)}^2\mbox{\quad as $\delta\to0$},
\end{eqnarray*}
which means that $v=\bar{v}\in  L^2([0,T];H^1_0(\Omega)^3)$. 

Thus, we conclude that $v\in L^2([0,T];\tilde{H}^1_{0,\sigma}(\Omega))=L^2([0,T];H^1_{0,\sigma}(\Omega))$.
\end{proof}
We show $t$-pointwise weak convergence of $\{\rho_\delta\}$, which is required in the next subsection. 
\begin{Prop}\label{pointwise}
There exists a sequence $\delta\to0$ and $\rho\in L^2([0,T];L^2(\tOmega))$ such that for every $t\in[0,T]$,
\begin{eqnarray*}
&&\rho^0_\ast \le\rho(t,\cdot)\le   \rho^0_{\ast\ast};\quad 
 \rho_\delta(t,\cdot)\wto\rho(t,\cdot)\mbox{ in $L^2(\tOmega)$ as $\delta\to0$}.
\end{eqnarray*} 
\end{Prop}   
\begin{proof}
We use an Ascoli-Arzela type reasoning. Set $\{s_k\}_{k\in\N}:=\Q\cap[0,T]$. Since $\{\rho_\delta(s_1,\cdot)\}$ is bounded in $L^2(\tOmega)$, there exists a subsequence   $\{\rho_{1l}\}_{l\in\N}\subset \{\rho_\delta\}$ and $\rho(\cdot;s_1)\in L^2(\tOmega)$ such that $\rho_{1l}(s_1,\cdot)\wto \rho(\cdot;s_1)$ in $L^2(\tOmega)$ as $l\to\infty$.   We check that $\rho^0_\ast\le\rho(\cdot;s_1)\le \rho^0_{\ast\ast}$: set $\tilde{\rho}(x):=\min\{ \rho(x;s_1)-\rho^0_{\ast},0 \}:  \tOmega\to\R_{\le0}$; since $\rho_{1l}(s_1,\cdot)-\rho^0_{\ast}\ge0$, we have $(\rho_{1l}(s_1,\cdot)-\rho^0_{\ast}, \tilde{\rho})_{L^2(\tOmega)}\le0$ for all $l$ and  $(\rho_{1l}(s_1,\cdot)-\rho^0_{\ast}, \tilde{\rho})_{L^2(\tOmega)}=(\rho_{1l}(s_1,\cdot), \tilde{\rho})_{L^2(\tOmega)}-(\rho^0_{\ast}, \tilde{\rho})_{L^2(\tOmega)}\to (\rho(\cdot;s_1)-\rho^0_{\ast}, \tilde{\rho})_{L^2(\tOmega)}=\norm\tilde{\rho}\norm_{L^2(\tOmega)}^2$ as $l\to\infty$; hence $\norm\tilde{\rho}\norm_{L^2(\tOmega)}^2\le0$ and $\tilde{\rho}=0$, i.e., $\rho_{1l}(s_1,\cdot)\ge\rho^0_{\ast}$; 
similarly, set $\tilde{\rho}(x):=\min\{ \rho^0_{\ast\ast}-\rho(x;s_1),0 \}:  \tOmega\to\R_{\le0}$; since $\rho^0_{\ast\ast}-\rho_{1l}(s_1,\cdot)\ge0$, we have $(\rho^0_{\ast\ast}-\rho_{1l}(s_1,\cdot), \tilde{\rho})_{L^2(\tOmega)}\le0$ for all $l$ and  $(\rho^0_{\ast\ast}-\rho_{1l}(s_1,\cdot), \tilde{\rho})_{L^2(\tOmega)}=(\rho^0_{\ast\ast}, \tilde{\rho})_{L^2(\tOmega)}-(\rho_{1l}(s_1,\cdot), \tilde{\rho})_{L^2(\tOmega)}\to (\rho^0_{\ast\ast}-\rho(\cdot;s_1), \tilde{\rho})_{L^2(\tOmega)}=\norm\tilde{\rho}\norm_{L^2(\tOmega)}^2$ as $l\to\infty$; hence $\norm\tilde{\rho}\norm_{L^2(\tOmega)}^2\le0$ and $\tilde{\rho}=0$, i.e., $\rho(s_1,\cdot)\le\rho^0_{\ast\ast}$. 

Since $\{\rho_{1l}(s_2,\cdot)\}_{l\in\N}$ is bounded in $L^2(\tOmega)$, there exists a subsequence   $\{\rho_{2l}\}_{l\in\N}\subset \{\rho_{1l}\}_{l\in\N}$ and $\rho(\cdot;s_2)\in L^2(\tOmega)$ such that $\rho_{2l}(s_2,\cdot)\wto \rho(\cdot;s_2)$ in $L^2(\tOmega)$ as $l\to\infty$, where $\rho^0_\ast\le\rho(\cdot;s_2)\le \rho^0_{\ast\ast}$. Repeating this process, we obtain a subsequence   $\{\rho_{k+1l}\}_{l\in\N}\subset \{\rho_{kl}\}_{l\in\N}$ and $\rho(\cdot;s_{k+1})\in L^2(\tOmega)$ such that $\rho_{k+1l}(s_{k+1},\cdot)\wto \rho(\cdot;s_{k+1})$ in $L^2(\tOmega)$ as $l\to\infty$, where $\rho^0_\ast\le\rho(\cdot;s_{k+1})\le \rho^0_{\ast\ast}$, for each $k\in\N$.  It is clear that $\{\rho_k\}_{k\in\N}$, $\rho_{k}:=\rho_{kk}$ satisfies 
$$\rho_k(s_{k'},\cdot)\wto \rho(\cdot;s_{k'})\mbox{\quad in $L^2(\tOmega)$ as $k\to\infty$, $\forall\,k'\in\N$}.$$

In order to see weak convergence of $\{\rho_k(t,\cdot)\}_{k\in\N}$ for all $t\in[0,T]$, we check ``equi-continuity'' of $\{(\rho_k(t,\cdot),\phi)_{L^2(\tOmega)}\}_{k\in\N}$ with respect to $t\in[0,T]$ for each fixed $\phi \in C^1_0(\tOmega)$. 
Let  $h_k,\tau_k,\eta^n_k$ etc., denote the quantities that  generate the step function $\rho_k$. 
 There exists $K(\phi)\in \N$ such that $\phi\equiv0$ on $C_{4h_k}(y)$ for all $y\in \p\tOmega_{h_k}$ and  for all $k\ge K(\phi)$. 
 If $k\ge K(\phi)$, the solution $\eta^n_k$ satisfies
\begin{eqnarray*}
&&(\eta_k^{n+1},\phi)_{\tOmega_{h_k}}
=\sum_{x\in\tOmega_{h_k}\setminus\p\tOmega_{h_k}}\Big( \frac{1}{7}\sum_{\omega\in B}\eta_k^n(x+h_k\omega)\phi(x)-D\cdot(\tilde{u}_k^n\eta_k^n)(x)\phi(x)\tau_k\Big)h_k^3\\
&&\quad =(\eta_k^{n},\phi)_{\tOmega_{h_k}}+
\sum_{x\in\tOmega_{h_k}\setminus\p\tOmega_{h_k}}\tilde{u}_k^n(x)\eta_k^n(x)\cdot D\phi(x)h_k^3\tau_k+O(h_k),
\end{eqnarray*}
where $|O(h_k)|\le M_0 h_k$ with a constant $M_0\ge0$  independent from $k$, $n$. Hence, we have  
\begin{eqnarray*}
&&|(\eta_k^{n+1},\phi)_{\tOmega_{h_k}}-(\eta_k^{n},\phi)_{ \tOmega_{h_k}}|\le\rho^0_{\ast\ast}\norm D\phi(x) \norm_{2,\tOmega_{h_k}}\norm \tilde{u}_k^n\norm_{2,\tOmega_{h_k}}\tau_k+M_0h_k\\
&&\quad \le \rho^0_{\ast\ast}\norm D\phi(x) \norm_{2,\tOmega_{h_k}}\norm u_k^n\norm_{2,\tOmega_{h_k}}\tau_k+M_0h_k.
\end{eqnarray*} 
For any  $0\le t<\tilde{t}\le T$, set $n_k,\tilde{n}_k\in\N\cup\{0\}$ so that $0\le n_k\le\tilde{n}_k\le T_{\tau_k}-1$, $t\in(\tau_k n_k,\tau_k n_k+\tau_k]$ and $\tilde{t}\in(\tau_k \tilde{n}_k, \tau_k \tilde{n}_k+\tau_k]$ if $t>0$ and $n_k=0$ if $t=0$. It follows from \eqref{3es1} that there exists a constant $M_1\ge0$ independent from $k$, $t$, $\tilde{t}$ such that  if $n_k<\tilde{n}_k$, 
\begin{eqnarray*}
&&|(\rho_k(\tilde{t},\cdot)-\rho_k(t,\cdot),\phi)_{L^2(\tOmega)}|
= |(\eta_k^{\tilde{n}_k+1},\phi)_{\tOmega_{h_k}}-(\eta_k^{n_k+1},\phi)_{\tOmega_{h_k}}|+O(h_k)\\
&&\quad\le \rho^0_{\ast\ast}\norm D\phi(x) \norm_{2,\tOmega_{h_k}} \sum_{n=n_k+1}^{\tilde{n}_k}\norm u_k^{n}\norm_{2,\tOmega_{h_k}}\tau_k +M_0h_k (\tilde{n}_k-n_k)\tau_k   +O(h_k)\\
&&\quad \le M_1(|\tilde{t}-t|+2\tau_k)+M_1h_k, 
\end{eqnarray*}
which includes the case of $n_k=\tilde{n}_k$ because $|(\rho_k(\tilde{t},\cdot)-\rho_k(t,\cdot),\phi)_{L^2(\tOmega)}|=0$ if  $n_k=\tilde{n}_k$. 
\indent  Fix an arbitrary small $\ep>0$. There exists $\delta(\ep)>0$ and $K(\ep)\in\N$ such that 
$$k,k'\ge K(\ep)\Longrightarrow  M_1(\delta(\ep)+2\tau_k)+M_1h_k+M_1(\delta(\ep)+2\tau_{k'})+M_1h_{k'}<\frac{2\ep}{3}.$$
Let $I_0:=[0,\delta(\ep)],I_1:=[\delta(\ep),2\delta(\ep)],\ldots,I_{J(\ep)}:=[{J(\ep)}\delta(\ep),T]$. Take a rational number $\tilde{s}_j$ from each $I_j$, $0\le j\le J(\ep)$ ($0\le j\le J(\ep)-1$ if $J(\ep)\delta(\ep)=T$). For any $t\in [0,T]$, there exists $I_j$ such that $t\in I_j$. 
Since $\{(\rho_k(\tilde{s}_j,\cdot),\phi)_{L^2(\tOmega)}\}_{k\in\N}$ is s convergent sequence of $\R$, there exists $K_j(\ep)\ge K(\phi)$ such that if $k,k'\ge K_j(\ep)$ we have 
 $$|(\rho_{k'}(\tilde{s}_j,\cdot),\phi)_{L^2(\tOmega)}-(\rho_k(\tilde{s}_j,\cdot),\phi)_{L^2(\tOmega)}|< \frac{\ep}{3}.$$
 Set $\tilde{K}(\ep)\in\N$ as 
 $$\tilde{K}(\ep):=\max\{K(\ep),K_0(\ep),K_1(\ep)\ldots,K_{J(\ep)}(\ep)\}.$$
Then, we  have for any $k,k'\ge \tilde{K}(\ep)$, 
\begin{eqnarray*}
&&|(\rho_{k'}(t,\cdot),\phi)_{L^2(\tOmega)}-(\rho_k(t,\cdot),\phi)_{L^2(\tOmega)}|
\le |(\rho_{k'}(t,\cdot),\phi)_{L^2(\tOmega)}-(\rho_{k'}(\tilde{s}_j,\cdot),\phi)_{L^2(\tOmega)}|\\
&&\quad +|(\rho_{k'}(\tilde{s}_j,\cdot),\phi)_{L^2(\tOmega)}-(\rho_k(\tilde{s}_j,\cdot),\phi)_{L^2(\tOmega)}|
+ |(\rho_{k}(\tilde{s}_j,\cdot),\phi)_{L^2(\tOmega)}-(\rho_k(t,\cdot),\phi)_{L^2(\tOmega)}|\\
&&\quad< M(\delta(\ep)+2\tau_k)+Mh_k+M(\delta(\ep)+2\tau_{k'})+Mh_{k'}+\frac{\ep}{3}
<\ep.
\end{eqnarray*} 
Therefore, $\{(\rho_k(t,\cdot),\phi)_{L^2(\tOmega)}\}_{k\in\N}$ is a convergent sequence of $\R$. 
On the other hand, since $\{\rho_k(t,\cdot)\}_{k\in\N}$ is bounded in $L^2(\tOmega)$, we have a subsequence $\{\tilde{\rho}_k(t,\cdot)\}_{k\in\N}\subset \{\rho_k(t,\cdot)\}_{k\in\N}$ and $\rho(\cdot;t)\in L^2(\tOmega)$ such that  $\rho^0_\ast\le\rho(\cdot;t)\le \rho^0_{\ast\ast}$ and 
$$\tilde{\rho}_k(t,\cdot)\wto \rho(\cdot;t)\mbox{\quad in $L^2(\tOmega)$ as $k\to\infty$},$$
which implies that 
$$\lim_{k\to\infty} (\rho_k(t,\cdot),\phi)_{L^2(\tOmega)}=\lim_{k\to\infty} (\tilde{\rho}_k(t,\cdot),\phi)_{L^2(\tOmega)} =(\rho(\cdot;t),\phi)_{L^2(\tOmega)},\quad\forall\,\phi\in C^1_0(\tOmega).$$
Since $C^1_0(\tOmega)$ is dense in $L^2(\tOmega)$,  we conclude that   $\rho_k(t,\cdot)\wto \rho(\cdot;t)$ in $L^2(\tOmega)$ as $k\to\infty$ for every $t\in [0,T]$. 
\end{proof}
In the rest of paper, $\{\rho_\delta\}$, $\{v_\delta\}$, $\{\rho_\delta v_\delta\}$ are  the sequences that satisfy the weak convergence shown in Proposition \ref{weak-convergence} and Proposition \ref{pointwise}. 
\subsection{Strong convergence of $\{v_\delta\}$}

Our aim is to prove that the pair of $\rho$ and $v$ found in Proposition \ref{weak-convergence} is a weak solution of \eqref{NS}. For this purpose, we prove  $L^2$-strong convergence of $\{v_\delta\}$ to $v$ through the following steps taken in \cite{Kuroki-Soga}, which  can be seen as a version of well-known Aubin-Lions-Simon approach: 
\begin{itemize}
\item[(S1)] Suppose that the weakly convergent sequence $\{ v_\delta \}$ obtained in Proposition \ref{weak-convergence}, which is re-denoted by $\{v_m\}_{m\in\N}$ ($\{ \rho_\delta \}$ is also re-denoted by $\{ \rho_m \}_{m\in\N}$),  is not strongly convergent in $L^2([0,T];L^2(\Omega)^3)$, i.e.,  $\{ v_m \}$ is not a Cauchy sequence in $L^2([0,T];L^2(\Omega)^3)$.
\item[(S2)] Then, there exists $\ep_0>0$ such that for each $m\in\N$ we have $k(m),l(m)\ge m$ for which 
$0<\ep_0\le \norm v_{k(m)}-v_{l(m)} \norm_{L^2([0,T];L^2(\Omega)^3)}$
 holds.
\item[(S3)] We will see that $\norm v_{k(m)}-v_{l(m)}  \norm_{L^2([0,T];L^2(\Omega)^3)}$ is bounded from the above by two different ``norms''.
\item[(S4)] We are able to estimate  the  ``norms'' to tend to $0$ as $m\to\infty$, only with the information on  the discrete time-derivative of $\rho_mv_m$ and weak convergence of $\{\rho_mv_m\}$, and we reach a contradiction. 
\end{itemize} 
As we will see later, once $L^2$-strong convergence of $\{v_\delta\}$ is proven, 
 we also obtain $L^2$-strong convergence of $\{\rho_\delta\}$ to $\rho$.    

The Aubin-Lions lemma (see, e.g., Lemma 2.1 in Section 2 of Chapter III, \cite{Temam-book}) is standard in this kind of arguments. Kuroki-Soga \cite{Kuroki-Soga} modified Aubin-Lions lemma in the convergence proof for Chorin's projection method applied to the homogeneous incompressible Navier-Stokes equations  so that reasoning similar to Aubin-Lions-Simon approach works under the discrete divergence-free constraint that depends on $\delta$. Here, we further modify  Kuroki-Soga's approach  (our current discrete problem provides the solution $u^{n}$ that is (discrete) $L^2_tH^1_x$-bounded and divergence-free, while Chorin's projection method provide the solution that is (discrete) $L^2_tH^1_x$-bounded but only ``asymptotically'' divergence-free). 

In the case of constant density problems, the modified  Aubin-Lions-Siom approach applied to the sequence $\{v_\delta\}$ refers to  the discrete time-derivative of $v_\delta$, which is treated with the discrete Navier-Stokes equations. However, in the case of non-constant density problems, the controllable quantity is the discrete time-derivative of $\rho_\delta v_\delta$. Because of this, we must  further modify  the Aubin-Lions type interpolation inequality  so that the discrete time-derivative of $\rho_\delta v_\delta$ can be involved in the weak norm.

We provide the ``norms'' mentioned in (S3) and state the interpolation inequality.  Let $\tau_{k(m)},h_{k(m)}, \eta^{n+1}_{k(m)}, u^{n+1}_{k(m)}$ (resp. $\tau_{l(m)},h_{l(m)}, \eta^{n+1}_{l(m)}, u^{n+1}_{l(m)}$), etc.,  be the quantities that provide the step functions $\rho_{k(m)}$, $v_{k(m)}$ (resp. $\rho_{l(m)}$, $v_{l(m)}$). 
For each $t\in[0,T]$, take $n_{k(m)},n_{l(m)}\in\N$ such that $t\in(\tau_{k(m)}n_{k(m)},\tau_{k(m)}n_{k(m)}+\tau_{k(m)}]$, $t\in(\tau_{l(m)}n_{l(m)},\tau_{l(m)}n_{l(m)}+\tau_{l(m)}]$ if $t>0$  and $n_{k(m)}=n_{l(m)}=0$ if $t=0$; define  
\begin{eqnarray*}
&& \nnorm v_{k(m)} (t,\cdot)\nnorm:= \Big(\norm u_{k(m)}^{n_{k(m)}+1}\norm_{2,\Omega_{h_{k(m)}}}^2+\sum_{j=1}^3\norm D^+_j u_{k(m)}^{n_{k(m)}+1}\norm_{2,\Omega_{h_{k(m)}}}^2\Big)^\2,\\
&& \nnorm v_{l(m)} (t,\cdot)\nnorm:= \Big(\norm u_{l(m)}^{n_{l(m)}+1}\norm_{2,\Omega_{h_{l(m)}}}^2+\sum_{j=1}^3\norm D^+_j u_{l(m)}^{n_{l(m)}+1}\norm_{2,\Omega_{h_{l(m)}}}^2\Big)^\2,\\
&&\nnorm \rho_{k(m)}(t,\cdot)v_{k(m)}(t,\cdot)-\rho_{l(m)}(t,\cdot)v_{l(m)}(t,\cdot)\nnorm_{\rm op} \\
&&\qquad \qquad\quad    := \sup_\phi
 \Big|(\eta^{n_{k(m)}+1}_{k(m)}u_{k(m)}^{n_{k(m)}+1}, \phi)_{\Omega_{h_{k(m)}}}-(\eta^{n_{l(m)}+1}_{l(m)}u_{l(m)}^{n_{l(m)}+1},  \phi)_{\Omega_{h_{l(m)}}}\Big|,
\end{eqnarray*}
where {\it the supremum is taken over all $\phi\in  C^3_{0,\sigma}(\Omega)$ such that $\norm \phi \norm_{W^{3,\infty}(\Omega)^3} =1$ and {\rm supp}$(\phi)
\cap \Omega_{h_{k(m)}} \subset \Omega_{h_{k(m)}}^{\circ}$, {\rm supp}$(\phi)
\cap \Omega_{h_{l(m)}} \subset \Omega_{h_{l(m)}}^{\circ}$ }; $(\cdot,\phi)_{\Omega_h}$ means $(\cdot,\phi|_{\Omega_h})_{\Omega_h}$. 
%
\begin{Lemma}\label{key-lemma}
For each $\nu>0$, there exists $A_\nu>0$  independent of  $t\in[0,T]$ such that 
\begin{eqnarray}\label{key}
&&\,\,\, \norm v_{k(m)}(t,\cdot)-v_{l(m)}(t,\cdot)\norm_{L^2(\Omega)^3}\le \nu (\nnorm v_{k(m)}(t,\cdot)\nnorm+\nnorm v_{l(m)}(t,\cdot)\nnorm+m^{-1})\\\nonumber
&&\quad+A_\nu (\nnorm \rho_{k(m)}(t,\cdot)v_{k(m)}(t,\cdot)-\rho_{l(m)}(t,\cdot)v_{l(m)}(t,\cdot)\nnorm_{\rm op}+m^{-1}), \quad \mbox{$\forall\,m\in\N$, $\forall\,t\in[0,T]$}.
\end{eqnarray} 
\end{Lemma}
\medskip

\noindent{\bf Remark.} {\it The presence of $m^{-1}$ would play an important role when we possibly have $\nnorm\cdot\nnorm_{\rm op}=0$, where it is not a priori clear $\nnorm\cdot\nnorm_{\rm op}\neq0$ or not. Kuroki-Soga \cite{Kuroki-Soga} dealt with the case where $\rho_{k(m)}\equiv\rho_{l(m)}\equiv1$, but they missed the regularization  by $m^{-1}$.  The presence of $m^{-1}$ does not change anything in regards to our application of Lemma \ref{key-lemma} to a proof of strong convergence. } 

\medskip
\begin{proof}
First we find $A_\nu$ for each fixed $t\in[0,T]$.  
Suppose that the assertion does not hold. 
Then, there exists some constant $\nu_0>0$ such that  for each $i\in\N$ we can find $m=m(i)\in\N$ such that  
\begin{eqnarray}\label{6161}
&&\norm v_{k(m(i))}(t,\cdot)-v_{l(m(i))}(t,\cdot)\norm_{L^2(\Omega)^3}> \nu_0 (\nnorm v_{k(m(i))}(t,\cdot)\nnorm +\nnorm v_{l(m(i))}(t,\cdot)\nnorm+m(i)^{-1}) \\\nonumber
&&\qquad+i (\nnorm  \rho_{k(m(i))}(t,\cdot)v_{k(m(i))}(t,\cdot)-\rho_{l(m(i))}(t,\cdot)v_{l(m(i))}(t,\cdot)\nnorm_{\rm op}+m(i)^{-1}),
\end{eqnarray}
where $m(i)$ cannot stay finite as $i\to\infty$ due to the presence of $m(i)^{-1}$ and we may assume  $m(i)\nearrow\infty$ as $i\to\infty$.   
Normalize $v_{k(m(i))}(t,\cdot),v_{l(m(i))}(t,\cdot)$ as   
\begin{eqnarray*}
&&\omega_{i}^1:=\frac{v_{k(m(i))}(t,\cdot)}{\nnorm v_{k(m(i))}(t,\cdot)\nnorm+\nnorm v_{l(m(i))}(t,\cdot)\nnorm+m(i)^{-1}},\\
&&\omega_{i}^2:=\frac{v_{l(m(i))}(t,\cdot)}{\nnorm v_{k(m(i))}(t,\cdot)\nnorm+\nnorm v_{l(m(i))}(t,\cdot)\nnorm+m(i)^{-1}}, 
\end{eqnarray*}
where $\omega_{i}^1$ and $\omega_{i}^2$ are still step functions defined on $\Omega$. 
Setting $\tilde{\omega}_{i}^1:=\omega_{i}^1|_{\Omega_{h_{k(m(i))}}}$, $\tilde{\omega}_{i}^2:=\omega_{i}^2|_{\Omega_{h_{l(m(i))}}}$ (restriction on the grid),  we see that  
\begin{eqnarray*}
&& \norm \tilde{\omega}^1_{i}\norm_{\Omega_{h_{k(m(i))}}}\le1,\,\,\, \norm \tilde{\omega}^2_{i}\norm_{\Omega_{h_{l(m(i))}}}\le1,\\
&&\norm D_j^+ \tilde{\omega}^1_{i}\norm_{\Omega_{h_{k(m(i))}}}\le1,\,\,\,
\norm D_j^+ \tilde{\omega}^2_{i}\norm_{\Omega_{h_{l(m(i))}}}\le1,\quad j=1,2,3.
\end{eqnarray*}
Let $\bar{\omega}^1_{i},\bar{\omega}^2_{i}:\Omega\to\R^3$ be the Lipschitz interpolation of $\omega^1_{i}$, $\omega^2_{i}$, respectively, by means of Lemma \ref{inter}.  We have 
\begin{eqnarray}\label{616161}
&& \norm \bar{\omega}^1_{i}-\omega^1_{i}\norm_{L^2(\Omega)^3}\le Kh_{k(m(i))},
\,\,\, \norm \bar{\omega}^2_{i}-\omega^2_{i}\norm_{L^2(\Omega)^3}\le Kh_{l(m(i))},\\\nonumber
&&\norm \partial_{x_j} \bar{\omega}^1_{i}\norm_{L^2(\Omega)^3} \le K',\,\,\,
\norm \partial_{x_j} \bar{\omega}^2_{i}\norm_{L^2(\Omega)^3} \le K', \mbox{ \quad  $\forall\,i\in\N$,  $j=1,2,3$},
\end{eqnarray}
where $K,K'$ are some constants. Hence, $\{\bar{\omega}^1_{i}\}_{i\in\N}$, $\{\bar{\omega}^2_{i}\}_{i\in\N}$ are bounded sequences of $H^1_0(\Omega)^3$;  
with reasoning similar to the proof of Proposition \ref{weak-convergence}, we find functions $\bar{\omega}^1,\bar{\omega}^2\in H^1_0(\Omega)^3$ such that $ \bar{\omega}^1_{i}\wto \bar{\omega}^1$, $ \bar{\omega}^2_{i}\wto \bar{\omega}^2$ in $H^1_0(\Omega)^3$ as $i\to\infty$  (up to a subsequence), 
as well as  $\partial_{x_j}\bar{\omega}^1_{i}\wto \partial_{x_j}\bar{\omega}^1$, $\partial_{x_j}\bar{\omega}^2_{i}\wto \partial_{x_j}\bar{\omega}^2$ in $L^2(\Omega)^3$ as $i\to\infty$ (up to a subsequence). 
On the other hand, due to the Rellich-Kondrachov theorem, taking a subsequence if necessary, we see that $ \bar{\omega}^1_{i}\to \bar{\omega}^1$, $ \bar{\omega}^2_{i}\to \bar{\omega}^2$ strongly in $L^2(\Omega)^3$ as $i\to\infty$. By \eqref{616161}, we have 
\begin{eqnarray}\label{6015}
\mbox{$\omega^1_{i}\to \bar{\omega}^1$,\,\, $\omega^2_{i}\to \bar{\omega}^2$ strongly in $L^2(\Omega)^3$ as $i\to\infty$.}
 \end{eqnarray}
Since $\tilde{\omega}^1_{i}$, $\tilde{\omega}^2_{i}$ are  discrete divergence-free,  we have  for each $\phi\in C^\infty_0(\Omega)$ (restricted to the grid) and for sufficiently large $i$, 
\begin{eqnarray*}
&&0=(D\cdot \tilde{\omega}^1_{i},\phi)_{\Omega_{h_{k(m(i))}}}
=-(\tilde{\omega}^1_{i},D\phi)_{\Omega_{h_{k(m(l))}}}
=-(\omega^1_{i},\nabla\phi)_{L^2(\Omega)^3} +O(h_{k(m(i))})\\
&&\quad \to -(\bar{\omega}^1,\nabla\phi)_{L^2(\Omega)^3} =(\nabla\cdot \bar{\omega}^1,\phi)_{L^2(\Omega)^3}=0
\mbox{ as $i\to\infty$}\mbox{ \quad (the same to  $\tilde{\omega}^2_{i}$)}
\end{eqnarray*}
to conclude that $\bar{\omega}^1,\bar{\omega}^2,\bar{\omega}:=\bar{\omega}^1-\bar{\omega}^2\in \tilde{H}^1_{0,\sigma}(\Omega)=H^1_{0,\sigma}(\Omega)$.

It follows from \eqref{6161} that 
\begin{eqnarray}\label{key1}
&&2\ge\norm \omega^1_{i}- \omega^2_{i} \norm_{L^2(\Omega)^3}> \nu_0 +i \nnorm  \rho_{k(m(i))}(t,\cdot) \omega^1_{i}-\rho_{l(m(i))}(t,\cdot) \omega^2_{i}\nnorm_{\rm op}\\\nonumber 
&&\qquad +i\frac{m(i)^{-1}}{\nnorm v_{k(m(i))}(t,\cdot)\nnorm+\nnorm v_{l(m(i))}(t,\cdot)\nnorm+m(i)^{-1}}
\ge\nu_0>0,\mbox{ $\forall\,i\in\N$},
\end{eqnarray}
which implies that 
\begin{eqnarray}\label{key2}
&& \nnorm  \rho_{k(m(i))}(t,\cdot) \omega^1_{i}-\rho_{l(m(i))}(t,\cdot) \omega^2_{i}\nnorm_{\rm op}\to 0\mbox{\quad as $i\to\infty$.}
\end{eqnarray} 
 For each $\phi\in C^3_{0,\sigma}(\Omega)$ with  $\norm \phi \norm_{W^{3,\infty}(\Omega)^3} =1$ and for all sufficiently large $i$, we obtain with Proposition \ref{pointwise} and \eqref{6015}, 
\begin{eqnarray*}
&&\nnorm  \rho_{k(m(i))}(t,\cdot) \omega^1_{i}-\rho_{l(m(i))}(t,\cdot) \omega^2_{i}\nnorm_{\rm op}\\
&&\quad \ge \Big|( \rho_{k(m(i))}(t,\cdot) \omega^1_{i},  \phi)_{\Omega_{h_{k(m(l))}}}
-(\rho_{l(m(i))}(t,\cdot) \omega^2_{i},  \phi)_{\Omega_{h_{l(m(l))}}}\Big|\\
&&\quad = \Big|( \rho_{k(m(i))}(t,\cdot),  \bar{\omega}^1  \phi)_{\Omega_{h_{k(m(l))}}}
+( \rho_{k(m(i))}(t,\cdot), (\omega^1_{i}- \bar{\omega}^1)  \phi)_{\Omega_{h_{k(m(l))}}}\\
&&\qquad -(\rho_{l(m(i))}(t,\cdot) , \bar{\omega}^2 \phi)_{\Omega_{h_{l(m(l))}}}
-(\rho_{l(m(i))}(t,\cdot) (\omega^2_{i}-\bar{\omega}^2)  \phi)_{\Omega_{h_{l(m(l))}}}
\Big|\\
&&\qquad \to\Big|    (\rho(t,\cdot),\bar{\omega}^1\phi)_{L^2(\Omega)^3}-  (\rho(t,\cdot),\bar{\omega}^2\phi)_{L^2(\Omega)^3} \Big|\mbox{ as $i\to\infty$}.
\end{eqnarray*}
Hence, with \eqref{6015},  (\ref{key1}) and (\ref{key2}),  we obtain      
$$ 0<\nu_0\le\norm \bar{\omega}\norm_{L^2(\Omega)^3},\quad(\rho(t,\cdot)\bar{\omega},\phi)_{L^2(\Omega)^3}=0,\mbox{ \,\,\,$\forall\,\phi\in  C^3_{0,\sigma}(\Omega)$.}$$
The first inequality implies $\bar{\omega}\neq0$. However,  since $\bar{\omega}\in H^1_{0,\sigma}(\Omega)$, we take $\{\omega_l\}_{l\in\N}\subset C^\infty_{0,\sigma}(\Omega)$ that approximates $\bar{\omega}$ in the $H^1$-norm as $l\to\infty$ and find    
\begin{eqnarray*}
\int_\Omega \rho(t,x)|\bar{\omega}(x)|^2dx&=&(\rho(t,\cdot)\bar{\omega},\bar{\omega})_{L^2(\Omega)^3}=(\rho(t,\cdot)\bar{\omega},\omega_l)_{L^2(\Omega)^3}+(\rho(t,\cdot)\bar{\omega},\bar{\omega}-\omega_l)_{L^2(\Omega)^3}\\
&=&(\rho(t,\cdot)\bar{\omega},\bar{\omega}-\omega_l)_{L^2(\Omega)^3}\to0\mbox{\quad as $l\to\infty$}.
\end{eqnarray*}
Since $0<\rho^0_{\ast}\le\rho(t,\cdot)\le \rho^0_{\ast\ast}$, we have $\bar{\omega}=0$, which is  a contradiction.  Therefore, there exists $A_\nu=A_\nu(t)>0$ for each $t\in[0,T]$ as claimed. 

We prove that  there exists  $A_\nu>0$ independent of the choice of  $t\in[0,T]$.  Fix any $\nu>0$. 
Let $A^\ast_\nu(t)$ be the infimum of $\{A_\nu\,|\, \mbox{\eqref{key} holds} \}$ for each fixed $t$. We will show  that $A^\ast_\nu(\cdot)$ is bounded on $[0,T]$. 
Suppose that  $A^\ast_\nu(\cdot)$ is not bounded. Then, we find a sequence $\{s_i\}_{i\in\N}\subset[0,T]$ for which $A_\nu^\ast(s_i)\nearrow\infty$ as $i\to\infty$.  Set $a_i:=A^\ast_\nu(s_i)/2$. For each $i\in \N$, there exists $m(i)\in\N$ for which we have   
\begin{eqnarray*}
 &&\norm v_{k(m(i))}(s_i,\cdot)-v_{l(m(i))}(s_i,\cdot)\norm_{L^2(\Omega)^3}> \nu (\nnorm v_{k(m(i))}(s_i,\cdot)\nnorm+\nnorm v_{l(m(i))}(s_i,\cdot)\nnorm+m(i)^{-1})\\
&&\qquad +a_i (\nnorm \rho_{k(m(i))}(s_i,\cdot)v_{k(m(i))}(s_i,\cdot)- \rho_{l(m(i))}(s_i,\cdot)v_{l(m(i))}(s_i,\cdot)\nnorm_{\rm op}+m(i)^{-1}).
\end{eqnarray*}
Note that $a_i\nearrow\infty$ as $i\to\infty$ and $\{s_i\}$ converges to some $t^\ast\in[0,T]$ as $i\to\infty$ (up to a subsequence); $m(i)$ cannot stay finite as $i\to\infty$ due to the presence of $m(i)^{-1}$. Since $\{m(i)\}_{i\in\N}$ is unbounded, we may follow the same reasoning as the first half of our proof  and reach a contradiction. In fact, we obtain the limit function $\bar{\omega}=\bar{\omega}^1-\bar{\omega}^2$ such that $0<\nu\le\norm\bar{\omega}\norm_{L^2(\Omega)^3}$ in the same way; we also obtain $(\rho(t^\ast,\cdot)\bar{\omega},\phi)_{L^2(\Omega)^3}=0$ for all $\phi\in  C^3_{0,\sigma}(\Omega)$ by 
 \begin{eqnarray*}
&&\nnorm  \rho_{k(m(i))}(s_i,\cdot) \omega^1_{i}-\rho_{l(m(i))}(s_i,\cdot) \omega^2_{i}\nnorm_{\rm op}\to 0\mbox{ as $i\to\infty$},\\
&&\nnorm  \rho_{k(m(i))}(s_i,\cdot) \omega^1_{i}-\rho_{l(m(i))}(s_i,\cdot) \omega^2_{i}\nnorm_{\rm op}\\
&&\quad \ge \Big|( \rho_{k(m(i))}(s_i,\cdot) \omega^1_{i},  \phi)_{\Omega_{h_{k(m(l))}}}
-(\rho_{l(m(i))}(s_i,\cdot) \omega^2_{i},  \phi)_{\Omega_{h_{l(m(l))}}}\Big|\\
&&\quad = \Big|\Big( \rho_{k(m(i))}(t^\ast,\cdot),  \bar{\omega}^1  \phi\Big)_{\Omega_{h_{k(m(l))}}}
+\Big( \rho_{k(m(i))}(s_i,\cdot)- \rho_{k(m(i))}(t^\ast,\cdot),  \bar{\omega}^1\phi\Big)_{\Omega_{h_{k(m(l))}}}\\
&&\qquad +\Big(\rho_{k(m(i))}(s_i,\cdot), (\omega_i^1-\bar{\omega}^1) \phi\Big)_{\Omega_{h_{k(m(l))}}}
-\Big( \rho_{l(m(i))}(t^\ast,\cdot),  \bar{\omega}^2  \phi\Big)_{\Omega_{h_{l(m(l))}}}\\
&&\qquad -\Big( \rho_{l(m(i))}(s_i,\cdot)- \rho_{l(m(i))}(t^\ast,\cdot),  \bar{\omega}^2\phi\Big)_{\Omega_{h_{l(m(l))}}}
 -\Big(\rho_{l(m(i))}(s_i,\cdot), (\omega_i^2-\bar{\omega}^2) \phi\Big)_{\Omega_{h_{l(m(l))}}}
\Big|\\
&&\qquad \to\Big|    (\rho(t^\ast,\cdot),\bar{\omega}^1\phi)_{L^2(\Omega)^3}-  (\rho(t^\ast,\cdot),\bar{\omega}^2\phi)_{L^2(\Omega)^3} \Big|\mbox{ as $i\to\infty$},
\end{eqnarray*}
where we use the ``equi-continuity'' shown in the proof of Proposition \ref{pointwise} with smooth approximation of $\bar{\omega}^1$ and $\bar{\omega}^2$. 
\end{proof}
\begin{Thm}\label{strong-convergence}
The sequence $\{v_\delta\}$ mentioned in Proposition  \ref{weak-convergence}, which is weakly convergent to the weak limit $v$, converges to $v$ strongly in $L^2([0,T];L^2(\Omega)^3)$.    
\end{Thm}
\begin{proof}
Re-write $\{\rho_\delta\}$, $\{v_\delta\}$  as $\{\rho_m\}_{m\in\N}$, $\{v_m\}_{m\in\N}$.  Suppose that $\{v_m\}$ does not converge to $v$ strongly in $L^2([0,T];L^2(\Omega)^3)$ as $m\to\infty$. Then, $\{v_m\}$ is not a Cauchy sequence in $L^2([0,T];L^2(\Omega)^3)$, i.e.,  there exists $\ep_0>0$ such that  for each $m\in \N$ there exist $k(m),l(m)\ge m$ for which  
 $0<\ep_0\le \norm v_{k(m)}-v_{l(m)}\norm_{L^2([0,T];L^2(\Omega)^3)}$ holds.  It follows from Lemma \ref{key-lemma} that  
\begin{eqnarray*}
0&<&\ep_0\le \norm v_{k(m)}-v_{l(m)}\norm_{L^2([0,T];L^2(\Omega)^3)}\\
&\le& \underline{\nu \Big\{ \Big(\int_0^T\nnorm v_{k(m)}(t,\cdot)\nnorm^2 dt \Big)^\2+ \Big(\int_0^T\nnorm v_{l(m)}(t,\cdot)\nnorm^2 dt \Big)^\2  \Big\}+\nu m^{-1}T^\2+A_\nu m^{-1}T^\2}_{\rm(\ast)} \\
&&+A_\nu\Big(  \int_0^T\nnorm \rho_{k(m)}(t,\cdot)v_{k(m)}(t,\cdot)-\rho_{l(m)}(t,\cdot)v_{l(m)}(t,\cdot)\nnorm_{\rm op}^2dt  \Big)^\2, \mbox{\quad $\forall\,m\in\N$},
\end{eqnarray*} 
where $\nu>0$ is arbitrarily chosen, $A_\nu$ is  a constant and  
\begin{eqnarray*}
&&\int_0^{T}\!\!\!\nnorm v_{k(m)}(t,\cdot)\nnorm^2dt \le \!\!\!\!
\sum_{0\le n< T_{\tau_{k(m)}}} \!\!\!\!
\Big(\norm u_{k(m)}^{n_{k(m)}+1}\norm_{2,\Omega_{h_{k(m)}}}^2+\sum_{j=1}^3\norm D^+_j u_{k(m)}^{n_{k(m)}+1}\norm_{2,\Omega_{h_{k(m)}}}^2 \Big)\tau_{k(m)}, \\
&&\int_0^{T}\!\!\!\nnorm v_{l(m)}(t,\cdot)\nnorm^2dt \le \!\!\!\!
\sum_{0\le n< T_{\tau_{l(m)}}} \!\!\!\!
\Big(\norm u_{l(m)}^{n_{l(m)}+1}\norm_{2,\Omega_{h_{l(m)}}}^2+\sum_{j=1}^3\norm D^+_j u_{l(m)}^{n_{l(m)}+1}\norm_{2,\Omega_{h_{l(m)}}}^2 \Big)\tau_{l(m)}. 
 \end{eqnarray*}
Due to \eqref{3es1} and \eqref{3es2},  for any small $\ep>0$ we may chose $\nu=\nu(\ep)>0$ and $M(\ep)\in\N$ for which   $(\ast)<\ep$ holds for all $m\ge M(\ep)$. 
If we prove $\nnorm \rho_{k(m)}(t,\cdot)v_{k(m)}(t,\cdot)-\rho_{l(m)}(t,\cdot)v_{l(m)}(t,\cdot)\nnorm_{\rm op}\to0$ as $m\to\infty$ for each $t\in(0,T)$, we reach a contradiction and the proof is done. 

The next step starts with a discrete version of the following obvious equality for two functions: 
\begin{eqnarray*}
g(t)\tilde{g}(t)=\frac{1}{\tilde{t}-t}\int_t^{\tilde{t}} g(s)\tilde{g}(s)ds
+\frac{1}{\tilde{t}-t}\int_t^{\tilde{t}} (s-\tilde{t})\frac{d}{ds}\{g(s)\tilde{g}(s))\}ds.
\end{eqnarray*}  
Fix $t\in(0,T)$ arbitrarily. Let $n_{k(m)}\in\N$ be such that $t\in(\tau_{k(m)}n_{k(m)},\tau_{k(m)}n_{k(m)}+\tau_{k(m)}]$. For a fixed $\tilde{t}\in (t,T)$, let $\tilde{n}_{k(m)}$ be such that $\tilde{t}\in(\tau_{k(m)}\tilde{n}_{k(m)},\tau_{k(m)}\tilde{n}_{k(m)}+\tau_{k(m)}]$. 
  Note that $0<\tau_{k(m)}(\tilde{n}_{k(m)}-n_{k(m)})-\tau_{k(m)}\le \tilde{t}-t\le \tau_{k(m)}(\tilde{n}_{k(m)}-n_{k(m)})+\tau_{k(m)}$ for all sufficiently large $m$. We will later appropriately choose $\tilde{t}$ close enough to $t$.  Define   
\begin{eqnarray*}
a_{k(m)}&\!\!\!:=&\!\!\!\frac{1}{\tau_{k(m)}(\tilde{n}_{k(m)}-n_{k(m)})}\sum_{n=n_{k(m)}+1}^{\tilde{n}_{k(m)}} \eta_{k(m)}^{n+1}u_{k(m)}^{n+1} \tau_{k(m)},\\
b_{k(m)}&\!\!\!:=&\!\!\!\frac{1}{\tau_{k(m)}(\tilde{n}_{k(m)}-n_{k(m)})}\!\!\sum_{n=n_{k(m)}+1}^{\tilde{n}_{k(m)}}\!\!\!\!\!\tau_{k(m)}\{(n-1)-\tilde{n}_{k(m)}\} \frac{\eta_{k(m)}^{n+1}u_{k(m)}^{n+1}-\eta_{k(m)}^{n}u^{n}_{k(m)} }{\tau_{k(m)}} \tau_{k(m)}\\
&\!\!\!=&\!\!\!\frac{1}{\tilde{n}_{k(m)}-n_{k(m)}}\sum_{n=n_{k(m)}+1}^{\tilde{n}_{k(m)}}\!\!\!\!\!\!\Big[ (n-\tilde{n}_{k(m)}) \eta_{k(m)}^{n+1}u_{k(m)}^{n+1}- \{(n-1)-\tilde{n}_{k(m)}\}\eta_{k(m)}^{n}u^{n}_{k(m)} \Big]\\
&&\!\!\!\!\!\!-a_{k(m)},  
\end{eqnarray*}
which leads to 
$$\eta_{k(m)}^{n_{k(m)}+1}u_{k(m)}^{n_{k(m)}+1}=a_{k(m)}+b_{k(m)}.$$ 
We introduce $n_{l(m)}$, $\tilde{n}_{l(m)}$, $a_{l(m)}$ and  $b_{l(m)}$ in the same way with the same $t$ and $\tilde{t}$, to have $\eta_{l(m)}^{n_{l(m)}+1}u_{l(m)}^{n_{l(m)}+1}=a_{l(m)}+b_{l(m)}$. 
Observe that 
\begin{eqnarray*}
&&\nnorm\rho_{k(m)}(t,\cdot)v_{k(m)}(t,\cdot)-\rho_{l(m)}(t,\cdot)v_{l(m)}(t,\cdot)\nnorm_{\rm op}\\
&&\qquad = \sup_\phi\Big| (\eta_{k(m)}^{n_{k(m)}+1}u_{k(m)}^{n_{k(m)}+1},\phi)_{\Omega_{h_{k(m)}}} -  (\eta_{l(m)}^{n_{l(m)}+1}u_{l(m)}^{n_{l(m)}+1},\phi)_{\Omega_{h_{l(m)}}}\Big|,\\
&&\Big|  (\eta_{k(m)}^{n_{k(m)}+1}u_{k(m)}^{n_{k(m)}+1},\phi)_{\Omega_{h_{k(m)}}} -  (\eta_{l(m)}^{n_{l(m)}+1}u_{l(m)}^{n_{l(m)}+1},\phi)_{\Omega_{h_{l(m)}}}\Big|\\
&&\qquad \le\Big|(a_{k(m)},\phi)_{\Omega_{h_{k(m)}}}-(a_{l(m)},\phi)_{\Omega_{h_{l(m)}}}\Big|+\Big| (b_{k(m)},\phi)_{\Omega_{h_{k(m)}}}\Big|
+\Big|(b_{l(m)},\phi)_{\Omega_{h_{l(m)}}} \Big|.
\end{eqnarray*}
We check that $\sup_{m,\phi}| (b_{k(m)},\phi)_{\Omega_{h{k(m)}}}|$ can be arbitrarily small as $\tilde{t}\to t+$ within admissible function $\phi$ (noting again that $\phi\equiv0$ near $\partial\Omega_{h_{k(m)}}$ and $\partial\Omega_{h_{l(m)}}$), where we insert the discrete Navier-Stokes equations into the discrete time-derivative. Hereafter, $M_1,M_2,\ldots$ are some constants independent of $t$, $\tilde{t}$, $m$ and admissible functions $\phi$.   
With the discrete Navier-Stokes equations  \eqref{d3}, we have 
\begin{eqnarray*}
&&| (b_{k(m)},\phi)_{\Omega_{h_{k(m)}}}|\le \sum_{n=n_{k(m)}+1}^{\tilde{n}_{k(m)}} \Big|\Big(\frac{\eta_{k(m)}^{n+1} u_{k(m)}^{n+1}-\eta_{k(m)}^{n}u^{n}_{k(m)} }{\tau_{k(m)}},\phi\Big)_{\Omega_{h_{k(m)}}} \Big|\tau_{k(m)}\\
&&\le \underline{\sum_{n=n_{k(m)}+1}^{\tilde{n}_{k(m)}} \Big|\left(\Big(\eta_{k(m)}^{n+1} u_{k(m)}^{n+1}-\frac{1}{7}\sum_{\omega\in B}\eta_{k(m)}^{n}(\cdot+h_{k(m)}\omega)u^{n}_{k(m)}(\cdot+h_{k(m)}\omega)  \Big)\tau_{k(m)}^{-1},\phi\right)_{\Omega_{h_{k(m)}}} \!\!\!\!\!\!\Big|\tau_{k(m)}}_{R_0}\\
&&\quad+\underline{\sum_{n=n_{k(m)}+1}^{\tilde{n}_{k(m)}} \Big|\left(\Big(\frac{1}{7}\sum_{\omega\in B}\eta_{k(m)}^{n}(\cdot+h_{k(m)}\omega)u^{n}_{k(m)}(\cdot+h_{k(m)}\omega)-\eta_{k(m)}^{n} u_{k(m)}^{n}  \Big)\tau_{k(m)}^{-1},\phi\right)_{\Omega_{h_{k(m)}}} \!\!\!\!\!\!\Big|\tau_{k(m)},}_{R_1}\\
&&R_0\le  \underline{\sum_{n=n_{k(m)}+1}^{\tilde{n}_{k(m)}}  
\Big|\Big(D\cdot(\eta_{k(m)}^n\tilde{u}^n_{k(m)})u_{k(m)}^{n+1},\phi \Big)_{\Omega_{h_{k(m)}}}   \Big|\tau_{k(m)}  }_{R_2} \\
&&\qquad + \frac{1}{2} \underline{\sum_{n=n_{k(m)}+1}^{\tilde{n}_{k(m)}}  
\Big|\sum_{j=1}^3  \Big(
\big(
\eta_{k(m)}^n(\cdot-h_{k(m)}e^j)\tilde{u}^n_{k(m)}{}_j(\cdot-he^j)D_ju^{n+1}_{k(m)}(\cdot-h_{k(m)}e^j)}\\
&&\qquad\qquad  \underline{+
\eta_{k(m)}^n(\cdot+h_{k(m)}e^j)\tilde{u}^n_{k(m)}{}_j(\cdot+h_{k(m)}e^j)D_ju_{k(m)}^{n+1}(\cdot+h_{k(m)}e^j)\big),\phi
\Big)_{\Omega_{h_{k(m)}}}   \Big|\tau_{k(m)}  }_{R_3} \\
&&\qquad+  \underline{\sum_{n=n_{k(m)}+1}^{\tilde{n}_{k(m)}}\Big|  \sum_{i=1}^3\Big(D^-\cdot\big\{\mu(\eta^{n+1})\big(D^+u^{n+1}_i+D^+_iu^{n+1}\big)\big\},\phi_i\Big)_{\Omega_{h_{k(m)}}}   \Big|\tau_{k(m)}  }_{R_4}\\
&&\qquad +\underline{\sum_{n=n_{k(m)}+1}^{\tilde{n}_{k(m)}}  \Big|  \Big(\eta_{k(m)}^{n+1}f_{k(m)}^{n+1},\phi\Big)_{\Omega_{h_{k(m)}}}  \Big|\tau_{k(m)}  }_{R_5}
 +\underline{\sum_{n=n_{k(m)}+1}^{\tilde{n}_{k(m)}}  \Big|  \Big(Dq_{k(m)}^{n+1},\phi\Big)_{\Omega_{h_{k(m)}}}  \Big|\tau_{k(m)}  }_{R_6}.
\end{eqnarray*}
We estimate the terms $R_1$-$R_6$. Since 
\begin{eqnarray*}
&&\sum_{\omega\in B}\sum_{x\in\Omega_h}\eta_{k(m)}^{n}(x+h_{k(m)}\omega)u^{n}_{k(m)}(x+h_{k(m)}\omega)\phi(x)\\
&&=\sum_{x\in\Omega_h}\eta_{k(m)}^{n}(x)u^{n}_{k(m)}(x)\phi(x)+\sum_{j=1}^3\sum_{x\in\Omega_h}\{\eta_{k(m)}^{n}(x+h_{k(m)}e^j)u^{n}_{k(m)}(x+h_{k(m)}e^j)\\
&&\quad -\eta_{k(m)}^{n}(x-h_{k(m)}e^j)u^{n}_{k(m)}(x-h_{k(m)}e^j)\} \phi(x)\\
&&=\sum_{x\in\Omega_h}\eta_{k(m)}^{n}(x)u^{n}_{k(m)}(x)\phi(x)+\sum_{j=1}^3\sum_{x\in\Omega_h}\eta_{k(m)}^{n}(x)u^{n}_{k(m)}(x)\{ \phi(x-h_{k(m)}e^j)+\phi(x+h_{k(m)}e^j)\}\\
&&=7\sum_{x\in\Omega_h}\eta_{k(m)}^{n}u^{n}_{k(m)}\phi(x)+ \sum_{j=1}^3\sum_{x\in\Omega_h}\eta_{k(m)}^{n}(x)u^{n}_{k(m)}(x)O(h_{k(m)}^2), 
\end{eqnarray*}
\eqref{scale}, \eqref{e-estimate1} and \eqref{3es1} implies that 
$$R_1\le M_1(\tilde{t}-t)$$
Observe that 
\begin{eqnarray*}
R_2&=&\sum_{n=n_{k(m)}+1}^{\tilde{n}_{k(m)}}  
\Big|\Big(\eta_{k(m)}^n\tilde{u}^n_{k(m)},D(u_{k(m)}^{n+1}\cdot \phi) \Big)_{\Omega_{h_{k(m)}}}   \Big|\tau_{k(m)}\\
&\le& M_2\sum_{n=n_{k(m)}+1}^{\tilde{n}_{k(m)}}\sum_{j=1}^3
 \norm\tilde{u}_{k(m)}^{n}\norm_{2,\Omega_h} \norm D_j^+u_{k(m)}^n\norm_{2,\Omega_h}  \tau_{k(m)}\\
&&+ M_2\sum_{n=n_{k(m)}+1}^{\tilde{n}_{k(m)}}\norm\tilde{u}_{k(m)}^{n}\norm_{2,\Omega_h}\norm u_{k(m)}^n\norm_{2,\Omega_h}  \tau_{k(m)}.
\end{eqnarray*}
By \eqref{molli}, \eqref{3es1} and \eqref{3es2}, we obtain 
\begin{eqnarray*}
R_2&\le& M_3 \sum_{j=1}^3\Big(\sum_{n=n_{k(m)}+1}^{\tilde{n}_{k(m)}}  \norm D_j^+u_{k(m)}^n\norm_{2,\Omega_h}^2 \tau_{k(m)} \Big)^\2\Big(\sum_{n=n_{k(m)}+1}^{\tilde{n}_{k(m)}} 1^2 \tau_{k(m)} \Big)^\2 \\
&&+M_3\Big(\sum_{n=n_{k(m)}+1}^{\tilde{n}_{k(m)}} 1^2 \tau_{k(m)} \Big)^\2 \\
&\le&M_4\sqrt{\tilde{t}-t}.
\end{eqnarray*}
A similar reasoning yields 
\begin{eqnarray*}
R_3\le M_5\sqrt{\tilde{t}-t},\quad
R_4\le M_6\sqrt{\tilde{t}-t},\quad
R_5\le M_7\sqrt{\tilde{t}-t}.
\end{eqnarray*}
By \eqref{qqqq}, we obtain 
$$R_6\le M_8 h_{k(m)}^\alpha+M_8 h_{k(m)}\sqrt{\tilde{t}-t}.$$
Therefore, we see that for any (small) $\ep>0$ there exists $\tilde{M}(\ep)\in\N$ and $\tilde{t}>t$ such that $| (b_{k(m)},\phi)_{\Omega_{h_{k(m)}}}|<\ep$ for all $m\ge \tilde{M}(\ep)$ and all admissible $\phi$, which holds for $| (b_{l(m)},\phi)_{\Omega_{h_{l(m)}}}|$ as well. 
On the other hand, since $\{\rho_{k(m)}v_{k(m)}\}_{m\in\N}$ and  $\{\rho_{l(m)}v_{l(m)}\}_{m\in\N}$ weakly converge to $\tilde{v}$ as $m\to\infty$ due to Proposition \ref{weak-convergence}, we have
\begin{eqnarray*}
&&\!\!\!\!\!\!\!\!\!\!\Big|(a_{k(m)},\phi)_{\Omega_{h_{k(m)}}}-(a_{l(m)},\phi)_{\Omega_{h_{l(m)}}}\Big|=\Big|(a_{k(m)},\phi)_{\Omega_{h_{k(m)}}}-\frac{1}{\tilde{t}-t}  \int_t ^{\tilde{t}} (\tilde{v}(s,\cdot),\phi)_{L^2(\Omega)^3} ds \\
&&\quad 
+  \frac{1}{\tilde{t}-t}  \int_t ^{\tilde{t}} (\tilde{v}(s,\cdot),\phi)_{L^2(\Omega)^3} ds-(a_{l(m)},\phi)_{\Omega_{h_{l(m)}}}\Big|\\
&&\le  M_9\frac{\tau_{k(m)}+ \tau_{l(m)}}{(\tilde{t}-t)^2} +\Big| \frac{1}{\tilde{t}-t}  \int_t ^{\tilde{t}} (\rho_{k(m)}(s,\cdot)v_{k(m)}(s,\cdot)-\tilde{v}(s,\cdot),\phi)_{L^2(\Omega)^3} ds\Big| \\
&&\quad+ \Big| \frac{1}{\tilde{t}-t}  \int_t ^{\tilde{t}} (\rho_{l(m)}(s,\cdot)v_{l(m)}(s,\cdot)-\tilde{v}(s,\cdot),\phi)_{L^2(\Omega)^3} ds\Big|\quad\to0\mbox{\quad as $m\to\infty$,}
\end{eqnarray*}
where it is easy to check that the convergence is uniform within all admissible functions $\phi$. Thus, we conclude that $\nnorm v_{k(m)}(t,\cdot)-v_{l(m)}(t,\cdot)\nnorm_{\rm op}\to0$ as $m\to\infty$ for each $t\in(0,T)$ and we reach a contradiction.
\end{proof}  
\subsection{Strong convergence of $\{\rho_\delta\}$}         

Let  $v\in L^2([0,T];H^1_{0,\sigma}(\Omega))$ be the one mentioned in Proposition \ref{weak-convergence} and Theorem \ref{strong-convergence}. We extend $v$ to be $0$ outside $\Omega$, where the extended $v$ belongs to $L^2([0,T];H^1(\R^3)^3)$ and satisfies $\nabla\cdot v=0$. 

We first show that the step function generated by $\tilde{u}^n=\mathcal{A}^{k_n}_hu^n$ also strongly converges to $v$ in $L^2([0,T];L^2(\tOmega)^3)$. For this purpose, we prove that  $\mathcal{A}^{k_n}_h$ is such that 
\begin{eqnarray}\label{moll}
\lim_{\tau,h\to0+}\max_{0\le n\le T_\tau}{\rm vol}(A^{k_n}_h)=0.
\end{eqnarray}
If not, we find a constant $a>0$, sequences $\tau_m,h_m\to 0+$ as $m\to\infty$ and $n=n(m)\in\{0,1,\ldots,T_{\tau_m}\}$ for each $m\in\N$   such that ${\rm vol}(A^{k_{n(m)}}_{h_m})\ge a>0$. Then, we see that some $k<k_{n(m)}$ yields 
$$  \norm\mathcal{A}^{k}_{h_m}u^{n(m)}\norm_{\infty,\tOmega_{h_m}}>\frac{2}{7}h_{m}^{-1+\alpha},\quad {\rm vol}(A^{k}_{h_m})\ge \frac{a}{2},\quad \forall\,m\in\N. $$
Hence, there exists $x\in\tOmega_{h_m}$ such that 
\begin{eqnarray*}
\frac{2}{7}h_{m}^{-1+\alpha}<|\mathcal{A}^{k}_{h_m}u^{n(m)}(x)|
\le \frac{1}{ {\rm vol}(A^{k}_{h_m})}\sum_{y\in A^{k}_{h_m}\cap h_m \Z^3} |u^{n(m)}(x+y)|h^3\le \sqrt{\frac{2}{a}}\norm u^{n(m)}\norm_{2,\Omega_{h_m}}.
\end{eqnarray*}
This is a contradiction, since $\norm u^{n(m)}\norm_{2,\Omega_{h_m}}
$ is bounded independently from $m$.   

Recall that $u^{n}$ is extended to be $0$ outside $\Omega_h$ for  $\tilde{u}^n=\mathcal{A}^{k_n}_hu^n$; we consider the step function $v_\delta$ generated by the extended $u^{n}$. 
It follows from \eqref{moll} that for each fixed $\ep >0$ we have 
$\max_{0\le n\le T_\tau}{\rm diameter}(A^{k_n}_h)\le \ep$  for all sufficiently small $(\tau,h)$.
Observe that   
\begin{eqnarray*}
&&I_\delta^2:=\sum_{0\le n\le T_\tau}\sum_{x\in \tOmega_h}|\tilde{u}^{n}(x)-u^{n}(x)|^2h^3\tau\\
&&\quad = \sum_{0\le n\le T_\tau}\sum_{x\in \tOmega_h}\Big|  \frac{1}{ {\rm vol}(A^{k_{n}}_{h})}\sum_{y\in A^{k_{n}}_{h}\cap h \Z^3} (u^{n}(x+y)-u^{n}(x))h^3 \Big|^2h^3\tau\\
&&\quad \le \sum_{0\le n\le T_\tau}\sum_{x\in \tOmega_h}\Big|  \frac{1}{ {\rm vol}(A^{k_{n}}_{h})}\Big(\sum_{y\in A^{k_{n}}_{h}\cap h \Z^3} 1h^3\Big)^\2\\
&&\qquad \times\Big(\sum_{y\in A^{k_{n}}_{h}\cap h \Z^3} |u^{n+1}(x+y)-u^{n+1}(x)|^2h^3\Big)^\2 \Big|^2h^3\tau\\
&&\quad =\sum_{0\le n\le T_\tau}  \frac{1}{ {\rm vol}(A^{k_{n}}_{h})}\sum_{y\in A^{k_{n}}_{h}\cap h \Z^3}\sum_{x\in \tOmega_h} |u^{n}(x+y)-u^{n}(x)|^2h^3h^3\tau\\
&&\quad =  \sum_{0\le n < T_\tau}  \frac{1}{ {\rm vol}(A^{k_{n+1}}_{h})} \sum_{y\in A^{k_{n+1}}_{h}\cap h \Z^3}\norm v_\delta(t_{n+1},\cdot+y)-v_\delta(t_{n+1},\cdot)\norm_{L^2(\tOmega)^3}^2h^3\tau+O(\tau).
\end{eqnarray*}
Let $y^{n+1}\in A^{k_{n+1}}_{h}\cap h \Z^3$ achieve the supremum  
$$\sup_{y\in A^{k_{n+1}}_{h}\cap h \Z^3}\norm v_\delta(t_{n+1},\cdot+y)-v_\delta(t_{n+1},\cdot)\norm_{L^2(\tOmega)^3}.$$
 Define the step function $y_\delta(t):=y^{n+1}$, $t\in(n\tau,n\tau+\tau]$, where $|y_\delta(t)|\le \ep$. Then, we have
\begin{eqnarray*}
&&I_\delta^2 \le  \sum_{0\le n<T_\tau} \norm v_\delta(t_{n+1},\cdot+y^{n+1})-v_\delta(t_{n+1},\cdot)\norm_{L^2(\tOmega)^3}^2\tau+O(\tau)\\
&&\quad \le \int_0^T \norm v_\delta(t,\cdot+y_\delta(t))-v_\delta(t,\cdot)\norm_{L^2(\tOmega)^3}^2dt+O(\tau)\\
&&\quad \le  2\int_0^T \norm v_\delta(t,\cdot+y_\delta(t))-v(t,\cdot+y_\delta(t))\norm_{L^2(\tOmega)^3}^2dt\\
&&\qquad + 2 \int_0^T \norm v(t,\cdot+y_\delta(t))-v(t,\cdot)\norm_{L^2(\tOmega)^3}^2dt
+ 2\int_0^T \norm v(t,\cdot)-v_\delta(t,\cdot)\norm_{L^2(\tOmega)^3}^2dt\\
&&\qquad +O(\tau)\\
&&\quad\le 2 \int_0^T \norm v(t,\cdot+y_\delta(t))-v(t,\cdot)\norm_{L^2(\tOmega)^3}^2dt
+ 4\int_0^T \norm v(t,\cdot)-v_\delta(t,\cdot)\norm_{L^2(\tOmega)^3}^2dt\\
&&\qquad  +O(\tau).
\end{eqnarray*}
Fix an arbitrarily small $\ep_1>0$. We have $w\in C^\infty([0,T]\times\tOmega;\R^3)$ such that supp$(w)\subset (0,T)\times\tOmega$ and $\norm v-w\norm_{L^2([0,T];L^2(\tOmega)^3)}<\ep_1$, where $w$ is extended to be $0$ outside $[0,T]\times\tOmega$. Since $w$ is uniformly continuous, we may choose the above $\ep>0$ so that 
$$\max_{(t,x)\in [0,T]\times\tOmega,\,|y|\le\ep }|w(t,y+x)-w(t,x)|< \ep_1.$$
Observe that 
\begin{eqnarray*}
&& \int_0^T \norm v(t,\cdot+y_\delta(t))-v(t,\cdot)\norm_{L^2(\tOmega)^3}^2dt\\
&&\quad \le  2\int_0^T \norm v(t,\cdot+y_\delta(t))- w(t,\cdot+y_\delta(y))\norm_{L^2(\tOmega)^3}^2dt\\
&&\qquad + 2\int_0^T \norm w(t,\cdot+y_\delta(t))- w(t,\cdot)\norm_{L^2(\tOmega)^3}^2dt
 + 2\int_0^T \norm w(t,\cdot)- v(t,\cdot)\norm_{L^2(\tOmega)^3}^2dt\\
 &&\quad< 4 \ep_1^2+2{\rm vol}(\tOmega)T\ep_1^2\mbox{\,\,\,\, as  $\delta=(h,\tau)\to0$.}
\end{eqnarray*}
Therefore, we have 
$$\limsup_{\delta\to0} I_\delta^2 < 8 \ep_1^2+4{\rm vol}(\tOmega)T\ep_1^2.$$
Since $\ep_1>0$ is arbitrary, we conclude that 
\begin{eqnarray}\label{moll2}
I_\delta=\Big(\sum_{n=0}^{T_\tau}\sum_{x\in \tOmega_h}|\tilde{u}^{n}(x)-u^{n}(x)|^2h^3\tau\Big)^\2\to0 \quad\mbox{as $\delta=(\tau,h)\to0$}.
\end{eqnarray}
Furthermore, for all sufficiently small $(\tau,h)$ such that $\max_{0\le n\le T_\tau}{\rm diameter}(A^{k_n}_h)\le \epsilon_0/2$, where $\epsilon_0$ is the constant to compare $\Omega$ and $\tOmega$, we see that 
\begin{eqnarray}\label{boooo}
\tilde{u}^{n}(x)=0\mbox{\quad on $\p\tOmega_h$,\quad $0\le \forall\,n\le T_\tau$.}
\end{eqnarray}
\begin{Thm}\label{strong2}
The sequence $\{\rho_\delta\}$ mentioned in Proposition \ref{weak-convergence} and Proposition  \ref{pointwise}, which is weakly convergent to the weak limit $\rho$, converges to $\rho$ strongly in $L^2([0,T];L^2(\Omega))$. Furthermore, $\rho$ satisfies
\begin{eqnarray}\label{DP}
&&\int_0^T\int_\tOmega\Big( \rho(t,x)\p_t\varphi(t,x) +\rho(t,x)v(t,x)\cdot\nabla \varphi(t,x)\Big)\,dxdt+\int_\tOmega \rho^0(x)\varphi(0,x)dx=0,\\\nonumber
&&\mbox{$\forall\,\varphi\in C^\infty([0,T]\times\R^3;\R)$ with {\rm supp}$(\varphi)\subset[0,T)\times\R^3$ compact}.
\end{eqnarray}  
In particular, it holds that $\rho(t,x)=\rho^0_\ast$ a.e. on $\tOmega\setminus\Omega$.    
\end{Thm}
\begin{proof}
We convert \eqref{d1} into a weak form. First, we argue within the class of test functions $\varphi\in C^\infty([0,T]\times\tOmega;\R)$ with {\rm supp}$(\varphi)\subset[0,T)\times\tOmega$.  
Fix such an arbitrary test function $\varphi$.  Shifting $x$ to $x\mp h \omega$, we have for all sufficiently small $(\tau,h)$,  
\begin{eqnarray*}
&&\sum_{n=0}^{T_\tau-1}\sum_{x\in \tOmega_h} \Big(\eta^{n+1}(x)-  \frac{1}{7}\sum_{\omega\in B}\eta^n(x+h \omega)  \Big)\frac{1}{\tau} \varphi(t_n,x) h^3\tau\\
&&\quad =\sum_{n=0}^{T_\tau-1}\sum_{x\in \tOmega_h} \Big(\eta^{n+1}(x) \varphi(t_n,x) -  \eta^n(x)\frac{1}{7}\sum_{\omega\in B}\varphi(t_n,x-h \omega)  \Big)\frac{1}{\tau} h^3\tau\\
&&\quad=\sum_{n=0}^{T_\tau-1}\sum_{x\in \tOmega_h} \Big(\eta^{n+1}(x) \varphi(t_n,x) -  \eta^n(x)\varphi(t_n,x) +\eta^n(x)O(h^2)  \Big)\frac{1}{\tau} h^3\tau\\
&&\quad=\sum_{n=0}^{T_\tau-1}\sum_{x\in \tOmega_h} \Big(\eta^{n+1}(x) \varphi(t_{n+1},x) -  \eta^n(x)\varphi(t_n,x) \Big) h^3\\
&&\qquad -\sum_{n=0}^{T_\tau-1}\sum_{x\in \tOmega_h} \eta^{n+1}(x)\frac{ \varphi(t_{n+1},x)- \varphi(t_{n},x)}{\tau} h^3\tau
 +\sum_{n=0}^{T_\tau-1}\sum_{x\in \tOmega_h}\eta^n(x)O\Big(\frac{h^2}{\tau}\Big)  h^3\tau\\
 &&\quad =-\sum_{x\in \tOmega_h} \eta^0(x)\varphi(0,x)h^3
  -\sum_{n=0}^{T_\tau-1}\sum_{x\in \tOmega_h} \eta^{n+1}(x)\p_t\varphi(t_{n+1},x)h^3\tau+O(h^{\alpha}),
\end{eqnarray*}
where we also note that $\varphi\equiv0$ near $t=T$ and $\eta^n$ is bounded.   Similarly,  we have 
\begin{eqnarray*}
&&\sum_{n=0}^{T_\tau-1}\sum_{x\in \tOmega_h}  D\cdot(\eta^n\tilde{u}^n)(x)\varphi(t_n,x) h^3\tau
 =-\sum_{n=0}^{T_\tau-1}\sum_{x\in  \tOmega_h}  \eta^n(x)\tilde{u}^n(x)\cdot D\varphi(t_n,x) h^3\tau\\
&&\quad =- \sum_{n=0}^{T_\tau-1}\sum_{x\in  \tOmega_h}
\eta^n(x)\tilde{u}^n(x)\cdot\nabla\varphi(t_n,x)h^3\tau
 -\sum_{n=0}^{T_\tau-1}\sum_{x\in  \tOmega_h}
\eta^n(x)\tilde{u}^n(x)\cdot O(h^2)h^3\tau.
\end{eqnarray*}
Therefore,  the weak form of \eqref{d1} is 
\begin{eqnarray}\label{weak1}
&&\qquad 0=\sum_{x\in \tOmega_h} \eta^0(x)\varphi(0,x)h^3
  +\sum_{n=0}^{T_\tau-1}\sum_{x\in \tOmega_h} \eta^{n+1}(x)\p_t\varphi(t_{n+1},x)h^3\tau\\\nonumber
&&\quad  + \sum_{n=0}^{T_\tau-1}\sum_{x\in  \tOmega_h}
\eta^n(x)\tilde{u}^n(x)\cdot\nabla\varphi(t_n,x)h^3\tau
 +\sum_{n=0}^{T_\tau-1}\sum_{x\in  \tOmega_h}
\eta^n(x)\tilde{u}^n(x)\cdot O(h^2)h^3\tau+O(h^{\alpha}).
\end{eqnarray}
 It follows from the weak convergence of $\{\rho_\delta\}$ and strong convergence of $\{v_\delta\}$ in \eqref{weak1} together with \eqref{moll2} that  
\begin{eqnarray}\label{41414141}
\int_\tOmega \rho^0\varphi(0,\cdot)dx+\int_0^T\int_\tOmega \rho \p_t\varphi dxdt+\int_0^T\int_\tOmega \rho v\cdot\nabla \varphi dxdt=0, 
\end{eqnarray}
where we note that 
\begin{eqnarray*}
&&\sum_{x\in \tOmega_h} \eta^0(x)\varphi(0,x)h^3
=\sum_{x\in \tOmega_h} h^{-3}\int_{C_h^+(x)}\rho^0(y)dy\varphi(0,x)h^3\\
&&\quad =\sum_{x\in \tOmega_h} \int_{C_h^+(x)}\rho^0(y)\varphi(0,y)dy+\sum_{x\in \tOmega_h} \int_{C_h^+(x)}\rho^0(y)(\varphi(0,x)-\varphi(0,y))dy\\
&&\quad \to \int_{t\Omega}\rho^0(x)\varphi(0,x)dx\quad\mbox{ as $\delta\to0$.}
\end{eqnarray*}
We show that $\rho=\rho^0_{\ast}$ a.e. on $\tOmega\setminus\Omega$. Take any test function $\varphi(t,x)=F(t)G(x)$ such that supp$(F)\subset(0,T)$ and supp$(G)\subset\tOmega\setminus\bar{\Omega}$. Since $v\equiv0$ on  $\tOmega\setminus\Omega$,  \eqref{41414141} yields 
$$\int_{\tOmega\setminus\bar{\Omega}}\Big(\int_0^T\rho(t,x)F'(t)dt \Big)G(x)dx=0,$$
which implies that 
\begin{eqnarray*}
&&\int_0^T\rho(t,x)F'(t)dt=0\mbox{\quad a.e. $x\in\tOmega\setminus\bar{\Omega}$;}\\
&&\rho(t,x)\mbox{ is constant  with respect to $t\in[0,T]$ for a.e. fixed  $x\in\tOmega\setminus\bar{\Omega}$.}
\end{eqnarray*}
Hence, including $F$ such that $F(0)\neq0$, we have 
\begin{eqnarray*}
0&=&\int_0^T\int_{\tOmega\setminus\bar{\Omega}}\rho(t,x)F'(t)G(x)dxdt +\int_{\tOmega\setminus\bar{\Omega}}\rho^0(x)F(0)G(x)dx \\
&=&F(0)\int_{\tOmega\setminus\bar{\Omega}}(\rho^0(x)-\rho(t,x))G(x)dx,
\end{eqnarray*}
which implies that $\rho(t,x)=\rho^0(x)=\rho^0_\ast$ a.e. on $\tOmega\setminus\Omega$, where $\rho^0$ has been extended to be $\rho^0_\ast$ outside $\Omega$.   

Now, we extend the class of test functions as mentioned in \eqref{DP}. For any such test function $\varphi$, consider a smooth cut-off $\tilde{\varphi}$ of $\varphi$ with respect to the $x$-variable such that  {\rm supp}$(\tilde{\varphi})\subset[0,T)\times\tOmega$ and $\varphi=\tilde{\varphi}$ on $\Omega$. 
Since  $\rho(t,x)=\rho^0(x)=\rho^0_\ast$  and $v(t,x)=0$ a.e.  $(t,x)\in[0,T]\times(\tOmega\setminus\Omega)$, it holds that 
\begin{eqnarray}\nonumber
0&=&\int_\tOmega \rho^0\tilde{\varphi}(0,\cdot)dx+\int_0^T\int_\tOmega \rho \p_t\tilde{\varphi} dxdt+\int_0^T\int_\tOmega \rho v\cdot\nabla \tilde{\varphi} dxdt\\\nonumber
&=&\int_\Omega \rho^0\tilde{\varphi}(0,\cdot)dx+\int_0^T\int_\Omega \rho \p_t\tilde{\varphi} dxdt+\int_0^T\int_\Omega \rho v\cdot\nabla \tilde{\varphi} dxdt\\\nonumber
&&+\int_{\tOmega\setminus\Omega} \rho^0\tilde{\varphi}(0,\cdot)dx+\int_0^T\int_{\tOmega\setminus\Omega} \rho \p_t\tilde{\varphi} dxdt+\int_0^T\int_{\tOmega\setminus\Omega} \rho v\cdot\nabla \tilde{\varphi} dxdt\\\nonumber
&=&\int_\Omega \rho^0\varphi(0,\cdot)dx+\int_0^T\int_\Omega \rho \p_t\varphi dxdt+\int_0^T\int_\Omega \rho v\cdot\nabla \varphi dxdt,\\\label{DP2}
0&=&\int_{\tOmega\setminus\Omega} \rho^0\varphi(0,\cdot)dx+\int_0^T\int_{\tOmega\setminus\Omega} \rho \p_t\varphi dxdt+\int_0^T\int_{\tOmega\setminus\Omega} \rho v\cdot\nabla \varphi dxdt.
\end{eqnarray}  
Therefore, we see that 
\begin{eqnarray*}
&&\int_\tOmega \rho^0\varphi(0,\cdot)dx+\int_0^T\int_\tOmega \rho \p_t\varphi dxdt+\int_0^T\int_\tOmega \rho v\cdot\nabla \varphi dxdt=0.
\end{eqnarray*}  
\indent In order to prove that the weak convergence is in fact strong convergence, we use the fact that an $L^2([0,T];L^2(\tOmega))$-function $\rho$ satisfying \eqref{DP} conserves its $L^2(\Omega)$-norm, i.e., 
$$\norm \rho(t,\cdot)\norm_{L^2(\tOmega)}=\norm \rho^0\norm_{L^2(\tOmega)},\quad\forall\,t\in[0,T]. $$
This is shown in \cite{DiPerna-Lions} for  problems on the whole space; our current bounded domain case can be  reduced to the whole space case by $0$-extension of $\rho$ and $v$ (see Introduction of \cite{Soga});  Tenan \cite{Tenan} directly proved counterparts of  \cite{DiPerna-Lions} for problems on a bounded domain. The general property of weak convergence provides 
\begin{eqnarray}\label{4st}
\sqrt{T}\norm \rho^0\norm_{L^2(\tOmega)}=\norm \rho\norm_{L^2([0,T];L^2(\tOmega))}\le \liminf_{\delta\to0} \norm \rho_\delta\norm_{L^2([0,T];L^2(\tOmega))}.
\end{eqnarray}
On the other hand,  for all sufficiently small $(\tau,h)$ such that \eqref{boooo} holds, 
\eqref{e-estimate1} with $p=2$ leads to 
\begin{eqnarray}\nonumber 
&&\norm \eta^{n+1}\norm_{2,\tOmega_h}\le \norm \rho^0\norm_{L^2(\tOmega)}, \quad 1\le \forall\,n+1\le T_\tau,\\\label{4st2}
 &&\limsup_{\delta\to0} \norm \rho_\delta\norm_{L^2([0,T];L^2(\tOmega))}\le \sqrt{T} \norm \rho^0\norm_{L^2(\tOmega)}=\norm \rho\norm_{L^2([0,T];L^2(\tOmega))}.
\end{eqnarray}
\eqref{4st} and \eqref{4st2} conclude that $\{\rho_\delta\}$ converges to $\rho$ strongly in $L^2([0,T];L^2(\tOmega))$. 
\end{proof}
\begin{Cor}\label{ssssss}
It holds that $ \rho_\delta\to\rho$, $\mu\circ\rho_\delta\to\mu\circ\rho$ in $L^p([0,T]; L^p(\tOmega))$ as $\delta\to0$ for any $p\in[1,\infty)$.
\end{Cor}
\begin{proof}
Let $p\in[1,\infty)$ be arbitrary.  It follows from Theorem \ref{strong2} that there exists a subsequence $\{\rho_m\}_{m\in\N}\subset\{\rho_\delta\}$ such that $\rho_m(t,x)\to\rho(t,x)$ a.e. $(t,x)\in[0,T]\times\tOmega$ as $m\to\infty$. Since $\mu$ is continuous, we have  $|\mu(\rho_m(t,x))-\mu(\rho(t,x))|^p\to0$ a.e. $(t,x)\in[0,T]\times\tOmega$ as $m\to\infty$. Since $|\mu(\rho_m(t,x))-\mu(\rho(t,x))|^p\le (2\mu_{\ast\ast})^p$, Lebesgue's dominated convergence theorem shows that $\norm \mu\circ\rho_m-\mu\circ\rho\norm_{L^p([0,T]);L^p(\Omega)}\to0$ as $m\to\infty$. 
 If $\{\mu\circ\rho_\delta\}$ does not converge to $\mu\circ\rho$ in $L^p([0,T]; L^p(\tOmega))$ as $\delta\to0$, we have a constant $\ep>0$ and a subsequence  $\{\tilde{\rho}_m\}_{m\in\N}\subset\{\rho_\delta\}$ such that  $\norm \mu\circ\tilde{\rho}_m-\mu\circ\rho\norm_{L^p([0,T]);L^p(\Omega)}\ge \ep$ for all $m$; however,  $\{\tilde{\rho}_m\}_{m\in\N}$ still converges to $\rho$ strongly in $L^2([0,T];L^2(\tOmega))$ as $m\to\infty$ and we have  a subsequence  $\{\hat{\rho}_m\}_{m\in\N}\subset\{\tilde{\rho}_m\}_{m\in\N}$ such that $\hat{\rho}_m(t,x)\to\rho(t,x)$ a.e. $(t,x)\in[0,T]\times\tOmega$ as $m\to\infty$;  this causes a contradiction.  
$ \rho_\delta\to\rho$ follows from the above argument with $\mu=id$. 
 \end{proof}
\subsection{Convergence to a weak solution}
We prove the following theorem: 
\begin{Thm}
The pair $\rho,v$ of the  limits of $\{\rho_\delta\}$ and  $\{v_\delta\}$ is a weak solution of \eqref{NS}.  
\end{Thm}
\begin{proof}
It follows from \eqref{DP} and \eqref{DP2} that $\rho$ satisfies \eqref{1313}.  

Next, we show that $\rho$ and $v$ satisfy \eqref{weak-form-NS}. Note that $v$ belongs to $L^\infty([0,T];L^2(\Omega)^3)$, because $v_\delta \in L^\infty([0,T];L^2(\Omega)^3)$ has $\delta$-independent bound of $\norm v_\delta \norm_{L^\infty([0,T];L^2(\Omega)^3)}$ due to \eqref{3es1}. Take an arbitrary test function $\phi$ and consider sufficiently small $\delta$. We have 
\begin{eqnarray*}
&&\sum_{n=0}^{T_\tau-1}\sum_{x\in \Omega_h} \Big(\eta^{n+1}(x)u^{n+1}(x) - \frac{1}{7}\sum_{\omega\in B}\eta^n(x+h \omega)u^n(x+h \omega)  \Big)\frac{1}{\tau}\cdot \phi(t_n,x) h^3\tau\\
&&\quad =\sum_{n=0}^{T_\tau-1}\sum_{x\in \Omega_h} \Big(\eta^{n+1}(x)u^{n+1}(x)\cdot \phi(t_n,x) -  \eta^n(x)u^{n}(x)\cdot\frac{1}{7}\sum_{\omega\in B}\phi(t_n,x-h \omega)  \Big)\frac{1}{\tau} h^3\tau\\
&&\quad=\sum_{n=0}^{T_\tau-1}\sum_{x\in \Omega_h} \Big(\eta^{n+1}(x)u^{n+1}(x)\cdot \phi(t_n,x) -  \eta^n(x)u^{n}(x)\cdot\phi(t_n,x) \\
&&\qquad +\eta^n(x)u^{n}(x)\cdot O(h^2)  \Big)\frac{1}{\tau} h^3\tau\\
&&\quad=\sum_{n=0}^{T_\tau-1}\sum_{x\in \Omega_h} \Big(\eta^{n+1}(x)u^{n+1}(x)\cdot \phi(t_{n+1},x) -  \eta^n(x)u^{n}(x)\cdot\phi(t_n,x) \Big) h^3\\
&&\qquad -\sum_{n=0}^{T_\tau-1}\sum_{x\in \Omega_h} \eta^{n+1}(x)u^{n+1}(x)\cdot\frac{ \phi(t_{n+1},x)- \phi(t_{n},x)}{\tau} h^3\tau\\
&&\qquad  +\sum_{n=0}^{T_\tau-1}\sum_{x\in \Omega_h}\eta^n(x)u^{n}(x)\cdot O\Big(\frac{h^2}{\tau}\Big)  h^3\tau\\
 &&\quad =-\sum_{x\in \Omega_h} \eta^0(x)u^{0}(x)\cdot\phi(0,x)h^3
  -\sum_{n=0}^{T_\tau-1}\sum_{x\in \Omega_h} \eta^{n+1}(x)u^{n+1}(x)\cdot\p_t\phi(t_{n+1},x)h^3\tau\\
  &&\qquad +O(h^{\alpha})\\
  &&\quad =-\sum_{x\in \Omega_h} \eta^0(x)u^{0}(x)\cdot\phi(0,x)h^3-\int_0^T\int_\Omega\rho_\delta(t,x)v_\delta(t,x)\cdot \p_t\phi(t,x)dxdt+O(h^\alpha),
\end{eqnarray*}
where we note that $\phi\equiv0$ near $t=T$. Similarly, we have  
\begin{eqnarray*}
&&\sum_{n=0}^{T_\tau-1}\sum_{x\in \Omega_h}\sum_{j=1}^3\Big\{D_j(\eta^n\tilde{u}^n_j)(x)u^{n+1}(x) +\frac{1}{2}\Big(
\eta^n(x-he^j)\tilde{u}^n_j(x-he^j)D_ju^{n+1}(x-he^j)\\\nonumber
&&\qquad\qquad +
\eta^n(x+he^j)\tilde{u}^n_j(x+he^j)D_ju^{n+1}(x+he^j)\Big)\Big\}\cdot\phi(t_n,x) h^3\tau\\\nonumber
&&=\sum_{n=0}^{T_\tau-1}\sum_{x\in \Omega_h}\sum_{j=1}^3\Big\{
\eta^n(x)\tilde{u}^n_j(x)u^{n+1}(x-he^j)\phi(t_n,x-he^j)\\
&&-\eta^n(x)\tilde{u}^n_j(x)u^{n+1}(x+he^j)\phi(t_n,x+he^j)
+\frac{1}{2}\Big(\eta^n(x)\tilde{u}^n_j(x)u^{n+1}(x+he^j)\phi(t_n,x+he^j)\\
&&-\eta^n(x)\tilde{u}^n_j(x)u^{n+1}(x-he^j)\phi(t_n,x+he^j)
+\eta^n(x)\tilde{u}^n_j(x)u^{n+1}(x+he^j)\phi(t_n,x-he^j)\\
&&-\eta^n(x)\tilde{u}^n_j(x)u^{n+1}(x-he^j)\phi(t_n,x-he^j)\Big)
\big\}\frac{h^3\tau}{2h}\\
&&=-\sum_{n=0}^{T_\tau-1}\sum_{x\in \Omega_h}\sum_{j=1}^3\eta^n(x)\tilde{u}^n_j(x)\frac{u^{n+1}(x-he^j)+u^{n+1}(x+he^j)}{2}\cdot D_j\phi(t_n,x)h^3\tau\\
&&=-\sum_{n=0}^{T_\tau-1}\sum_{x\in \Omega_h}\sum_{j=1}^3\eta^n(x)\tilde{u}^n_j(x)\frac{u^{n+1}(x-he^j)+u^{n+1}(x+he^j)}{2}\cdot \p_{x_j}\phi(t_n,x)h^3\tau+O(h^2)\\
&&=-\sum_{j=1}^3\int_0^T\int_\Omega\rho_\delta(t,x)v_\delta{}_j(t,x)\frac{v_\delta(x-he^j)+v_\delta(x-he^j)}{2}\cdot\p_{x_j}\phi(t,x)dxdt\\
&&\quad -\sum_{n=0}^{T_\tau-1}\sum_{x\in \Omega_h}\sum_{j=1}^3\eta^n(x)(\tilde{u}^n_j(x)-u^n_j(x)) \frac{u^{n+1}(x-he^j)+u^{n+1}(x+he^j)}{2}\cdot \p_{x_j}\phi(t_n,x)h^3\tau\\&&\quad +O(h);
\end{eqnarray*}
\begin{eqnarray*}
&&\sum_{n=0}^{T_\tau-1}\sum_{x\in \Omega_h}\sum_{i=1}^3D^-\cdot\Big\{\mu(\eta^{n+1})\Big(D^+u^{n+1}_i+D^+_iu^{n+1}\Big)\Big\}(x)\phi_i(t_n,x)h^3\tau\\
&&\quad  =-\sum_{n=0}^{T_\tau-1}\sum_{x\in \Omega_h}\sum_{i=1}^3\mu(\eta^{n+1}(x))\Big(D^+u^{n+1}_i+D^+_iu^{n+1}\Big)\cdot D^+\phi_i(t_n,x)h^3\tau\\
&&\quad =-\sum_{j=1}^3\int_0^T\int_\Omega \mu(\rho_\delta(t,x)) \Big(w^j_\delta(t,x)\cdot \p_{x_j}\phi(t,x)
+w^j_\delta(t,x)\cdot \nabla\phi_j(t,x) \Big)dxdt+O(h),
\end{eqnarray*} 
 where $w_\delta^1,w_\delta^2,w_\delta^3$ are mentioned in Proposition \ref{weak-convergence}. Hence, the weak form of \eqref{d3} is 
 \begin{eqnarray*}
 &&0=\underline{\sum_{x\in \Omega_h} \eta^0(x)u^{0}(x)\cdot\phi(0,x)h^3}_{ R_1}\\
&& +\underline{\int_0^T\int_\Omega\rho_\delta(t,x)v_\delta(t,x)\cdot \p_t\phi(t,x)dxdt}_{R_2}\\
 &&+\underline{\sum_{j=1}^3\int_0^T\int_\Omega\rho_\delta(t,x)v_\delta{}_j(t,x)\frac{v_\delta(x-he^j)+v_\delta(x-he^j)}{2}\cdot\p_{x_j}\phi(t,x)dxdt}_{R_3}\\
&&-\underline{\sum_{j=1}^3\int_0^T\int_\Omega \mu(\rho_\delta(t,x)) \Big(w^j_\delta(t,x)\cdot \p_{x_j}\phi(t,x)
+w^j_\delta(t,x)\cdot \nabla\phi_j(t,x) \Big)dxdt}_{R_4}\\ 
&&+\underline{\sum_{n=0}^{T_\tau-1}\sum_{x\in \Omega_h}\eta^{n+1}(x)f^{n+1}(x)\cdot\phi(t_n,x)h^3\tau}_{R_5}\\
&&-\underline{\sum_{n=0}^{T_\tau-1}\sum_{x\in \Omega_h}Dq^{n+1}(x)\cdot\phi(t_n,x)h^3\tau}_{R_6}\\
&&+\underline{\sum_{n=0}^{T_\tau-1}\sum_{x\in \Omega_h}\sum_{j=1}^3\eta^n(x)(\tilde{u}^n_j(x)-u^n_j(x)) \frac{u^{n+1}(x-he^j)+u^{n+1}(x+he^j)}{2}\cdot \p_{x_j}\phi(t_n,x)h^3\tau}_{R_7}\\
&&+O(h^\alpha). 
\end{eqnarray*} 
We evaluate $R_1$-$R_7$. 
Hereafter, $M_1,M_2,\ldots$ are some constants independent of $\delta$. 
Observe that 
\begin{eqnarray*}
&&\eta^0(x)u^0(x)\cdot\phi(0,x)=h^{-3}\int_{C_h^+(x)}\rho^0(y)dy\times h^{-3}\int_{C_h^+(x)}v^0(y)dy\cdot\phi(0,x)\\
&&\quad=h^{-3}h^{-3}\int_{C_h^+(x)}\int_{C_h^+(x)}\rho^0(y)v^0(y')\cdot\phi(0,x)dydy'\\
&&\quad=h^{-3}\int_{C_h^+(x)}\rho^0(y)v^0(y)\cdot\phi(0,x)dy\\
&&\qquad +h^{-3}h^{-3}\int_{C_h^+(x)}\int_{C_h^+(x)}\rho^0(y)(v^0(y')-v^0(y))\cdot\phi(0,x)dydy'\\
&&\quad=h^{-3}\int_{C_h^+(x)}\rho^0(y)v^0(y)\cdot\phi(0,y)dy +h^{-3}\int_{C_h^+(x)}\rho^0(y)v^0(y)\cdot O(h)dy \\
&&\qquad +h^{-3}h^{-3}\int_{C_h^+(x)}\int_{C_h^+(x)}\rho^0(y)(v^0(y')-v^0(y))\cdot\phi(0,x)dydy'.
\end{eqnarray*}
Let $\{v^0_m\}_{m\in\N}\in C^1_0(\Omega)$ be an approximating sequence of $v^0$ in $L^2(\Omega)^3$. We have 
\begin{eqnarray*}
(\ast)&:=&\Big|\sum_{x\in\Omega_h}h^{-3}h^{-3}\int_{C_h^+(x)}\int_{C_h^+(x)}\rho^0(y)(v^0(y')-v^0(y))\cdot\phi(0,x)dydy'h^3\Big|\\
& =&\Big|\sum_{x\in\Omega_h}h^{-3}\int_{C_h^+(x)}\int_{C_h^+(x)}\rho^0(y)(v^0(y')-v^0_m(y'))\cdot\phi(0,x)dydy'\\
&&\qquad +\sum_{x\in\Omega}h^{-3}\int_{C_h^+(x)}\int_{C_h^+(x)}\rho^0(y)(v_m^0(y)-v^0(y))\cdot\phi(0,x)dydy'\\
&&\qquad+\sum_{x\in\Omega_h}h^{-3}\int_{C_h^+(x)}\int_{C_h^+(x)}\rho^0(y)(v_m^0(y')-v_m^0(y))\cdot\phi(0,x)dydy'\Big|\\
&\le& M_1\norm v^0-v^0_m\norm_{L^2(\Omega)^3}
+M_2\sum_{x\in\Omega_h}h^{-3}\int_{C_h^+(x)}\int_{C_h^+(x)}|v_m^0(y')-v_m^0(y)|dydy',
\end{eqnarray*}
where $M_1,M_2$ are independent from $m$. 
 For any $\ep>0$, fix $m$ so that $M_1\norm v^0-v^0_m\norm_{L^2(\Omega)^3}<\ep$. Since $v^0_m$ is uniformly continuous on $\Omega$, we have 
\begin{eqnarray*}
(\ast)\le \ep+M_2\sup_{|y-y'|\le \sqrt{3}h}|v_m^0(y')-v_m^0(y)|{\rm vol}(\Omega)\to\ep\mbox{\quad as $h\to0+$}.
\end{eqnarray*}  
Since $\ep>0$ is arbitrary, we conclude that   
\begin{eqnarray*}
R_1\to \int_0^T\int_\Omega \rho^0(t,x)v^0(x)\cdot\phi(0,x)dx\mbox{ \quad as $\delta\to0$}.
\end{eqnarray*}
Theorem \ref{weak-convergence}, Theorem \ref{strong-convergence}, Theorem \ref{strong2} and Corollary \ref{ssssss} yield 
\begin{eqnarray*}
&&R_2\to \int_0^T\int_\Omega \rho(t,x)v(t,x)\cdot\p_t\phi(t,x)dxdt\mbox{ \quad as $\delta\to0$},\\
&&R_3\to \sum_{j=1}^3\int_0^T\int_\Omega \rho(t,x)v_j(t,x)v(t,x)\cdot\p_{x_j}\phi(t,x)dxdt\mbox{ \quad as $\delta\to0$},\\
&&R_4\to \sum_{j=1}^3\int_0^T\int_\Omega \mu(\rho(t,x))\Big(\p_{x_j}v(t,x)\cdot\p_{x_j}\phi(t,x)
+\p_{x_j}v(t,x)\cdot\nabla\phi_j(t,x)\Big)dxdt\\
&&\qquad =  \sum_{j=1}^3\int_0^T\int_\Omega \mu(\rho(t,x))\Big(\p_{x_j}v(t,x)
+\nabla v_j(t,x)\Big)\cdot\p_{x_j}\phi(t,x)dxdt \mbox{ \quad as $\delta\to0$},
\end{eqnarray*}
where we note that $\rho_\delta v_\delta\to\rho v=\tilde{v}$ in $L^2([0,T];L^2(\Omega)^3)$ as $\delta\to0$. 
Observe that 
\begin{eqnarray*}
&&\eta^{n+1}(x)f^{n+1}(x)\cdot\phi(t_n,x)=\tau^{-1}h^{-3}\int^{\tau(n+1)}_{\tau n}\int_{C_h^+(x)}\rho_\delta(s,y)f(s,y) \cdot\phi(t_n,x)dyds\\
&&\quad =\tau^{-1}h^{-3}\int^{\tau(n+1)}_{\tau n}\int_{C_h^+(x)}\rho_\delta(s,y)f(s,y) \cdot\phi(s,y)dyds \\
&&\qquad +\tau^{-1}h^{-3}\int^{\tau(n+1)}_{\tau n}\int_{C_h^+(x)}\rho_\delta(s,y)f(s,y) \cdot O(h)dyds.
\end{eqnarray*}
Hence, we obtain 
$$R_5\to \int_0^T\int_\Omega \rho(t,x)f(t,x)\cdot\phi(t,x)dxdt\mbox{ \quad as $\delta\to0$}.$$
It follows from \eqref{qqqq} that $R_6\to 0$  as $\delta\to0$. 
\eqref{moll2} implies that $R_7\to 0$  as $\delta\to0$.
\end{proof}

\medskip\medskip\medskip

\noindent{\bf Acknowledgement.} 
This paper was written during author's one-year research stay  in Fachbereich Mathematik, Technische Universit\"at Darmstadt, Germany, with the grant Fukuzawa Fund (Keio Gijuku Fukuzawa Memorial Fund for the Advancement of Education and Research). The author expresses special thanks to Professor Dieter Bothe for his kind hosting in TU-Darmstadt. The author is supported by JSPS Grant-in-aid for Young Scientists \#18K13443 and JSPS Grants-in-Aid for Scientific Research (C) \#22K03391. 

\medskip\medskip\medskip


\begin{thebibliography}{9}
%
\bibitem{AK} S. N. Antontsev and A. V.  Kazhikhov,  Mathematical questions of the dynamics of nonhomogeneous fluids (Russian), Lecture notes, Novosibirsk State University (1973). 
%
\bibitem{Chorin} A. J. Chorin, On the convergence of discrete approximations to the Navier-Stokes equations, Math. Comp. {\bf 23} (1969), pp. 341-353. 
%
\bibitem{DM} R. Danchin and P. Mucha, The incompressible Navier-Stokes equations in vacuum, Comm. Pure Appl. Math. {\bf72} (2019), no. 7, pp. 1351-1385. 
%
\bibitem{DiPerna-Lions}  R. J. DiPerna  and P.-L. Lions,   Ordinary differential equations, transport theory and Sobolev spaces, Invent. Math. {\rm \bf 98} (1989), no. 3, pp. 511-547.  
%
\bibitem{Feireisl2}E. Feireisl, T. Karper and M. Mich\'alek, 
Convergence of a numerical method for the compressible Navier-Stokes system on general domains, Numer. Math. {\bf 134} (2016), No. 4, pp. 667-704. 
%
\bibitem{Gallouet2} T. Gallou\"et, R. Herbin, J.-C. Latch\'e and D. Maltese, Convergence of the MAC scheme for the compressible stationary Navier-Stokes equations, Math. Comp. {\bf 87} (2018), No. 311, pp. 1127-1163. 
%
\bibitem{GS}  J.-L. Guermond and  A. J.  Salgado,
Error analysis of a fractional time-stepping technique for incompressible flows with variable density, 
SIAM J. Numer. Anal. {\bf 49} (2011), no. 3, pp. 917-944.  
%
\bibitem{HS} R. Ho\v{s}ek, and B. She, 
Stability and consistency of a finite difference scheme for compressible viscous isentropic flow in multi-dimension,  
J. Numer. Math. {\bf 26} (2018), no. 3, pp. 111-140. 
%
\bibitem{Karper} T. K. Karper, A convergent FEM-DC method for the compressible Navier-Stokes equations, Numer. Math. {\bf 125} (2013), pp. 441-510.  
%
\bibitem{Kazhikhov} A.V.  Kazhikhov, 
Solvability of the initial-boundary value problem for the equations of the motion of an inhomogeneous viscous incompressible fluid (Russian),  
Dokl. Akad. Nauk SSSR {\bf 216} (1974), pp. 1008-1010. 
%
\bibitem{Kim} J. U. Kim, 
Weak solutions of an initial-boundary value problem for an incompressible viscous fluid with nonnegative density.
SIAM J. Math. Anal. {\bf 18} (1987), no. 1, pp. 89-96. 
%
\bibitem{KL}A. Krzywicki and O. A. Ladyzhenskaya, The method of nets for non-stationary Navier-Stokes equations, Trudy Mat. Inst. Steklov. {\bf 92} (1966), pp 93-99.
%
\bibitem{Kuroki-Soga} H. Kuroki and K. Soga, On convergence of Chorin's projection method to a Leray-Hopf weak solution, Numer. Math.  {\bf 146} (2020), pp. 401-433. 
%
\bibitem{Ladyzhenskaya} O. A. Ladyzhenskaya, {\it The  Mathematical Theory of Viscous  Incompressible Flow},  2nd English edition. Math. Appl., 2. Gordon and Breach, New York (1969).
%
\bibitem{JL} J.-L. Lions, 
On some questions in boundary value problems of mathematical physics. Contemporary developments in continuum mechanics and partial differential equations,  (Proc. Internat. Sympos., Inst. Mat., Univ. Fed. Rio de Janeiro, Rio de Janeiro, 1977), pp. 284-346,
North-Holland Math. Stud., 30, North-Holland, Amsterdam-New York (1978).
%
\bibitem{Lions} P. L.  Lions, {\it Mathematical topics in fluid mechanics. Vol. 1. Incompressible models}, Oxford Lecture Series in Mathematics and its Applications, 3. Oxford Science Publications. The Clarendon Press, Oxford University Press, New York (1996).
 %
 %
\bibitem{LW} C. Liu and N. J. Walkington, 
Convergence of numerical approximations of the incompressible Navier-Stokes equations with variable density and viscosity,  
SIAM J. Numer. Anal. {\bf 45} (2007), no. 3, pp. 1287-1304. 
 %
\bibitem{Maeda-Soga} M. Maeda and K. Soga, More on convergence of Chorin's projection method for incompressible Navier-Stokes equations, J. Math. Fluid Mech. {\bf 24} (2022), no. 2, paper no. 41. 
%
\bibitem{Simon} J. Simon, 
Nonhomogeneous viscous incompressible fluids: existence of velocity, density, and pressure, 
SIAM J. Math. Anal. {\bf 21} (1990), no. 5, pp. 1093-1117. 
%
\bibitem{Soga} K. Soga, Finite difference methods for linear transport equations, preprint (arXiv:2209.10594).
%
\bibitem{Temam-2} R. Temam,  
Sur l'approximation de la solution des \'equations de Navier-Stokes par la m\'ethode des pas fractionnaires. II. (French), 
Arch. Rational Mech. Anal. {\bf 33} (1969), pp. 377-385. 
%
\bibitem{Temam-book} R. Temam, 
{\it Navier-Stokes equations:   theory and numerical analysis},  North-Holland, Amsterdam (1979). 
%
\bibitem{Tenan} J. Tenan, Weak transport equation on a bounded domain: stability theory apr\'es DiPerna-Lions,  preprint (arXiv:2103.09695). 
\end{thebibliography}
\end{document}